\documentclass[11pt]{article}
\usepackage[a4paper,top=3cm,bottom=3.5cm,left=2.5cm,right=2.5cm,marginparwidth=1.75cm]{geometry}
\usepackage[T1]{fontenc}
\usepackage{graphicx} 
\usepackage{amsmath}
\usepackage{amssymb,amsthm}
\usepackage{matlab-prettifier}
\usepackage[normalem]{ulem}
\usepackage{algpseudocode}
\usepackage[title]{appendix}
\usepackage{subcaption}
\newtheorem{theorem}{Theorem}

\newtheorem{proposition}{Proposition}
\newtheorem{lemma}{Lemma}
\newtheorem{remark}{Remark}
\usepackage[citestyle=numeric-comp,maxbibnames=99,giveninits=true,doi=true,url=false,eprint=false,isbn=false]{biblatex}
\usepackage{hyperref}
\hypersetup{colorlinks,allcolors=black}
\newtheorem{definition}{Definition}
\usepackage{caption}
\usepackage{subcaption}
\usepackage{appendix}
\usepackage{algorithm}
\usepackage{algpseudocode}

\usepackage{siunitx}
\usepackage{booktabs}
\usepackage{amsmath}  

\newcommand{\abs}[1]{\left|{#1}\right|}

\title{Fourier--Galerkin Methods for Subwavelength Resonances in two-dimensional Acoustic Metamaterials}
\author{Jinghao Cao\thanks{\footnotesize Department of Mathematics, Imperial College London, SW7 2AZ, London, United Kingdom ( jinghao.cao@imperial.ac.uk)} }
\date{}
\begin{document}
\maketitle
\begin{abstract}
We present a Fourier--Galerkin asymptotic framework for the analysis and computation of subwavelength resonances in two-dimensional scattering problems in finite domains. Starting from the boundary integral formulation, we apply a Fourier--Galerkin discretization to derive an explicit finite-dimensional effective matrix whose kernel characterizes the resonant frequencies. In the subwavelength regime, we obtain asymptotic expansions of this matrix in terms of the operating frequency and the material contrast, identifying the leading-order operators and their kernel structure.

This reduction transforms the resonance problem into a low-dimensional nonlinear eigenvalue problem, avoiding large-scale discretizations and global root-search procedures. The entries of the effective matrix are explicitly computable and admit fast evaluation using FFT-based quadrature. The resulting approach provides an efficient and robust computational framework for resonances in general smooth geometries.
\end{abstract}
\tableofcontents
\section{Introduction}

The control and manipulation of wave-matter interactions at subwavelength scales has driven the rapid development of artificially structured metamaterials \cite{lemoult.kaina.ea2016Soda,yves.fleury.ea2017Crystalline,phononic1,phononic2,ammari2024functional}. By engineering media with locally resonant microstructures—composed of fundamental building blocks that are themselves subwavelength resonators—it is possible to achieve exotic and highly customizable wave phenomena \cite{ammari2017subwavelength,yves.fleury.ea2017Crystalline,yves2017topological,wang2019subwavelength}. A natural and highly effective route to designing such metamaterials is through high-contrast resonators, where the subwavelength nature arises from extreme material contrast. Such structures can be used to achieve a wide variety of effects \cite{davies2019fully,ammari.davies.ea2022Exceptional,ammari2020highfrequency,ammari2020highorder,ammari2017subwavelength,ammari2019double,ammari2020honeycomb,ammari2018mathematical}. Typical examples span multiple physical domains, from acoustic Minnaert resonances generated by air bubbles in water \cite{ammari.fitzpatrick.ea2018Minnaert} to electromagnetic wave confinement in high-contrast dielectric and plasmonic particles \cite{ammari.li.ea2023Mathematical,ammari.millien.ea2017Mathematical}.

There has been substantial effort dedicated to the rigorous asymptotic analysis of these systems, establishing a deep mathematical understanding of how microscopic resonances dictate macroscopic wave behavior. Building on this analytical foundation, recent theoretical and experimental advances have further highlighted the profound implications of symmetry breaking. By introducing geometric defects, for instance, eigenmodes can be localized to create topological edge modes \cite{erik_thesis,bryn_thesis,feppon.cheng.ea2023Subwavelength}. Similarly, introducing non-reciprocal coupling in non-Hermitian systems violates the conventional bulk-edge correspondence principle, triggering the skin effect where bulk eigenmodes condense at the boundaries \cite{ghatak.brandenbourger.ea2020Observation,skinadd1,skinadd2,skinadd3}. Furthermore, breaking time-reversal symmetry via time-modulated material parameters allows for dynamic wave control, yielding non-reciprocal band gaps and unidirectional edge modes \cite{ammari2020time,Ammari_nonrecip,ammari_cao_transmprop}. Several further two-dimensional studies concern periodic geometries \cite{ammari2020highfrequency,ammari2020honeycomb}, where the quasi-periodic Green's function does not exhibit the logarithmic degeneracy for non-zero quasi-momentum, and are therefore of a different nature. On the computational side, boundary integral and boundary element discretisations of acoustic transmission problems are classical \cite{colton_kress,nedelec2001acoustic,saranen2002periodic}, and for circular resonators Fourier expansions combined with Graf's addition theorem are a standard tool; see, e.g., the appendix of \cite{ammari2020modelling}.

Despite these advances, a significant challenge remains: most existing asymptotic frameworks are restricted to highly symmetric configurations, such as circular or spherical domains, which limits their applicability to real-world metamaterial design. The precise analysis and optimization of complex macroscopic phenomena continue to demand both rigorous asymptotic characterization and increasingly efficient computational methods that can move beyond these idealized geometries. This need is especially critical for engineering purposes, where the transition from theoretical models to robust, functional devices relies on the ability to predict and manipulate resonant behavior with high fidelity across general and complex domain geometries.

To address this need, we develop a Fourier--Galerkin framework for the analysis and computation of subwavelength resonances in two-dimensional scattering problems. Starting from a boundary integral formulation, we perform a systematic modal decomposition to derive an explicit finite-dimensional effective matrix equation governing the resonant frequencies. This reduction, rooted in rigorous asymptotic analysis, results in a framework that is both analytically revealing and computationally efficient. The formulation yields rigorous asymptotic expansions in the subwavelength regime and transforms the scattering problem into a numerically tractable nonlinear eigenvalue problem. Crucially, the effective matrix can be computed explicitly for general geometries, enabling robust resonance computations without the need for global root-search procedures. Combined with fast quadrature techniques, the proposed approach extends naturally to complex domains, bridging the gap between theoretical characterization and practical engineering design.

Ultimately, these results establish a unified analytical and computational framework connecting boundary integral equations, asymptotic analysis, and fast algorithms. By providing a rigorous tool to evaluate complex operator perturbations and localized wave interactions, this framework serves as an essential mathematical and numerical foundation for the analysis and design of next-generation subwavelength metamaterials. Especially, this work paves the way for two-dimensional non-Hermitian effects such as the skin effect \cite{ammari2025non,ammari2024mathematical,ammari2023edge} and time-dependent scattering \cite{ammari2026effective,10.1098/rspa.2024.0177,ammari2023transmission} which have so far mainly been investigated in one and three dimensions. 

This paper is organized as follows. In Section~\ref{sec:pre}, we introduce the setting of the subwavelength transmission problem. We recall the classical boundary integral formulations and the boundary element methods necessary to derive a Galerkin matrix representation of the two-dimensional subwavelength problem. In Section~\ref{sec:asymp}, we derive two reduced formulations via asymptotic analysis of~\eqref{eq:matrixA} in the limit $\delta \to 0$. These reductions preserve spectral convergence while significantly reducing the problem dimension and computational cost; we demonstrate how the reduced model provides accurate initial guesses for resonance refinement and an analytically tractable framework for classifying subwavelength branches. Finally, in Section~\ref{sec:num}, we present the numerical treatment for the FFT-based full boundary integral equation and the effective matrix implementation. We discuss strategies for implementing the effective Galerkin matrices for general smooth domains and provide resonance and mode simulations for circular, ring-shaped, and elliptical geometries.
\section{Preliminaries}
\label{sec:pre}
Let $D \subset \mathbb R^2$ be a union of disjoint bounded Lipschitz domains
\[
D = \bigcup_{i=1}^N D_i,
\qquad
\partial D = \bigcup_{i=1}^N \partial D_i,
\]
with outward unit normal $\nu$ defined component-wise. We consider the Helmholtz transmission problem: find $u \in H^1_{\mathrm{loc}}(\mathbb{R}^2)$ such that
\begin{equation}
\label{eq:helmholtz_multi}
\begin{cases}
\Delta u + \omega^2 u = 0, & \text{in } \mathbb{R}^2 \setminus \overline{D}, \\[0.3em]
\Delta u + \omega^2 u = 0, & \text{in } D_i,\quad i=1,\dots,N,
\end{cases}
\end{equation}
subject to the transmission conditions on each $\partial D_i$,
\begin{equation}
\label{eq:transmission_multi}
u^+ = u^-,
\qquad
\delta\,\partial_\nu u^+ - \partial_\nu u^- = 0 \quad \text{on } \partial D,
\end{equation}
and the Sommerfeld radiation condition
\begin{equation}
\label{eq:sommerfeld}
\lim_{|x|\to\infty} |x|^{1/2}\left(\partial_{|x|}u - i\omega u\right)=0.
\end{equation}
We adopt the convention that the superscripts $+$ and $-$ indicate traces taken from the exterior $\mathbb{R}^2\setminus\overline{D}$ and the interior $D$, respectively.

In this section, we recall classical boundary integral formulations and boundary element methods in order to derive a Galerkin matrix representation (Equation \eqref{eq:AN_full_entries}) of the two-dimensional subwavelength problem. 
Section~\ref{sec:BIE} follows the construction of subwavelength resonator models in \cite{ammari2024functional}, while Section~\ref{sec:BEM} presents the corresponding boundary element discretization (see \cite{colton_kress,nedelec2001acoustic}) applied to this setting.
\subsection{Boundary Integral Equations}
\label{sec:BIE}
We recall the classical boundary integral formulation of the transmission problem~\eqref{eq:helmholtz_multi}, following the construction of~\cite{ammari2024functional}.
Let
\begin{equation}
G_\omega(x,y)=-\frac{i}{4}H_0^{(1)}(\omega|x-y|)
\end{equation}
denote the Helmholtz fundamental solution in two dimensions. We represent the interior and exterior fields by single-layer potentials
\begin{equation}
\label{eq:single_layer_multi}
u_\pm(x)
=
\sum_{j=1}^N \mathcal S_{\partial D_j}^\omega[\varphi_j^\pm](x)
\quad \textrm{with} \quad 
\mathcal S_{\partial D_j}^\omega[\varphi](x)
=
\int_{\partial D_j}G_\omega(x,y)\,\varphi(y)\,d\sigma.
\end{equation}
We introduce global densities
\begin{equation}
\phi=(\phi_j)_{j=1}^N,
\qquad
\psi=(\psi_j)_{j=1}^N,
\end{equation}
defined on $\partial D=\bigcup_{j=1}^N \partial D_j$. The associated boundary integral operators act component-wise as
\begin{equation}
\label{eq:Somega}
(S_\omega \varphi)_i(x)
=
\sum_{j=1}^N \int_{\partial D_j}G_\omega(x,y)\,\varphi_j(y)\,d\sigma,
\qquad x\in\partial D_i,
\end{equation}
and
\begin{equation}
(K_\omega' \varphi)_i(x)
=
\sum_{j=1}^N \operatorname{p.v.}\!\int_{\partial D_j}
\partial_{\nu(x)}G_\omega(x,y)\,\varphi_j(y)\,d\sigma,
\qquad x\in\partial D_i.
\end{equation}
The single-layer potential satisfies the classical trace and jump relations (See \cite{nedelec2001acoustic})
\begin{equation}
\mathrm{tr}^\pm \mathcal S^\omega[\varphi]=S_\omega[\varphi],
\quad \textrm{and} \quad
\mathrm{tr}_{\nu}^\pm \mathcal S^\omega[\varphi]
=
\left(\pm \tfrac12 I + K_\omega'\right)[\varphi].
\end{equation}
Imposing the transmission conditions across $\partial D$ leads to the boundary integral system
\begin{equation}
\label{eq:block_system}
\begin{cases}
S_\omega(\phi-\psi)=0,\\[0.4em]
\delta\left(\tfrac12 I + K_\omega'\right)\phi
-
\left(-\tfrac12 I + K_\omega'\right)\psi
=0.
\end{cases}
\end{equation}
We therefore introduce the operator-valued function
\begin{equation}
\label{eq:matrixA}
\mathcal A(\omega,\delta)
=
\begin{pmatrix}
S_\omega & -S_\omega\\[0.4em]
\left(-\tfrac12 I + K_\omega'\right) & -\delta\left(\tfrac12 I + K_\omega'\right)
\end{pmatrix}.
\end{equation}

\begin{definition}
A complex number $\omega$ is called an \emph{eigenfrequency} (or \emph{resonance frequency}) of the transmission problem ~\eqref{eq:helmholtz_multi} if there exists $(\phi,\psi)\neq(0,0)$ such that
\begin{equation}
\mathcal A(\omega,\delta)
\begin{pmatrix}
\phi\\ \psi
\end{pmatrix}
=0.
\end{equation}
Further, let $\Omega\subset\mathbb C$ be a neighborhood of the origin. A {subwavelength resonance} is a continuous function $\omega(\delta)\in\Omega$ such that
\begin{equation}
\ker \mathcal A\bigl(\omega(\delta),\delta\bigr)\neq\{0\},
\end{equation}
for all sufficiently small $\delta>0$ which satisfies
\begin{equation}
\omega(\delta)\to 0 \quad \text{as } \delta\to 0.
\end{equation}
\end{definition}

\subsection{Fourier--Galerkin Boundary Element Formulation}
\label{sec:BEM}
In two dimensions, the logarithmic nature of the Helmholtz Green's function makes the subwavelength resonance problem non-standard, and the existence and enumeration of resonances require a careful operator-theoretic analysis. The results in~\cite{ammari2024functional} show that a two dimensional system of $N$ subwavelength resonators admits exactly $N$ subwavelength resonant frequencies, and provide asymptotic formulas for computing them. While the analysis there is complete at the operator level, it leads in two dimensions to matrix entries of considerable complexity: the formula for $\mathcal{A}^{(2)}_{\omega,\delta}$ involves logarithmic corrections, integrals of $(S^{(0)})^{-1}$, and the equilibrium densities $\psi_j$, all of which must be computed indirectly. Rather than following this route, we propose a direct Fourier--Galerkin discretization of the transmission system~\eqref{eq:block_system}, which avoids inverting $S^{(0)}$ entirely and yields matrix entries expressed as explicit, computable boundary integrals of the layer potential kernels. The connection between the resulting effective matrix and the classical capacitance matrix is made precise in Section~\ref{sec:asymp}.

Let each disjoint boundary component $\partial D_j$ admit a $2\pi$–periodic parametrization
\begin{equation}
x_j(\theta),\qquad \theta\in[0,2\pi).   
\end{equation}
Fix a truncation order $F\in\mathbb N$ and define the trigonometric polynomial space on each component
$\partial D_j$ 
\begin{equation}
    \mathcal T_F(\partial D_j):=
\mathrm{span}\{e^{in\theta}\;:\;|n|\le F\}.
\end{equation}
The direct sum gives the $N(2F+1)$ dimensional global approximation space
\begin{equation}
\mathcal X_F
=
\bigoplus_{j=1}^N \mathcal T_F(\partial D_j).   
\end{equation}
The densities are approximated by truncated Fourier expansions
\begin{equation}\label{eq:fourier_ansatz}
\phi_j(\theta)\approx
\phi_{j,F}(\theta)
=
\sum_{n=-F}^{F}\hat\phi_{j,n}e^{in\theta},
\qquad
\psi_j(\theta)\approx
\psi_{j,F}(\theta)
=
\sum_{n=-F}^{F}\hat\psi_{j,n}e^{in\theta},
\end{equation}
where the unknown coefficients are given by
\[
\hat\phi_{j,n},\ \hat\psi_{j,n}\in\mathbb C,
\qquad
j=1,\dots,N,\quad n=-F,\dots,F.
\]
Denote the vectors of unknowns by
\begin{equation}
    \widehat\phi :=
(\hat\phi_{1,-F},\dots,\hat\phi_{1,F},\ \dots,\ \hat\phi_{N,-F},\dots,\hat\phi_{N,F})
\in\mathbb C^{N(2F+1)},
\end{equation}
and
\begin{equation}
    \widehat\psi :=
(\hat\psi_{1,-F},\dots,\hat\psi_{1,F},\ \dots,\ \hat\psi_{N,-F},\dots,\hat\psi_{N,F})
\in\mathbb C^{N(2F+1)},
\end{equation}
Choose the same space for testing and use the $L^2$ inner product on each $\partial D_i$:
\begin{equation}
    \langle f,g\rangle_{\partial D_i}
:=
\int_{\partial D_i} f(\theta)\,\overline{g(\theta)}\,ds.
\end{equation}
Let the test functions on $\partial D_i$ be the Fourier modes
\begin{equation}
    \eta_{i,m}(\theta):=e^{im\theta},\qquad |m|\le F.
\end{equation}
The \emph{discrete problem} is thus to find $(\phi_F,\psi_F)\in\mathcal X_F\times\mathcal X_F$,
not both zero, such that for all $i=1,\dots,N$ and all $|m|\le F$,
\begin{equation}\label{eq:galerkin_eqs}
\Big\langle
\eta_{i,m},\ (S_\omega(\phi_F-\psi_F))_i
\Big\rangle_{\partial D_i}=0,
\qquad
\Big\langle
\eta_{i,m},\
\big((-\tfrac12 I+K_\omega')\phi_F-\delta(\tfrac12 I+K_\omega')\psi_F\big)_i
\Big\rangle_{\partial D_i}=0.
\end{equation}
Since $S_\omega$ and $K_\omega'$ are {linear} in the densities, inserting the expansions \eqref{eq:fourier_ansatz} into \eqref{eq:galerkin_eqs} yields a
finite system that is \emph{linear in the coefficients} $(\widehat\phi,\widehat\psi)$:
\begin{equation}\label{eq:matrix_pencil}
A_F(\omega)
\binom{\widehat\phi}{\widehat\psi}
=0,
\qquad
A_F(\omega)\in\mathbb C^{(2N(2F+1))\times (2N(2F+1))}
\end{equation}
with the matrix entries given by
\begin{equation}\label{eq:AN_full_entries}
\bigl(A_F(\omega)\bigr)_{(\ell,(i,m)),(r,(j,n))}
=
\begin{cases}
\displaystyle
\big\langle \eta_{i,m},\, (S_\omega[\eta_{j,n}])_i \big\rangle_{\partial D_i},
& \ell=1,\ r=1, \\[6pt]

\displaystyle
-\big\langle \eta_{i,m},\, (S_\omega[\eta_{j,n}])_i \big\rangle_{\partial D_i},
& \ell=1,\ r=2, \\[6pt]

\displaystyle
\big\langle \eta_{i,m},\, \bigl((-\tfrac12 I + K_\omega')[\eta_{j,n}]\bigr)_i \big\rangle_{\partial D_i},
& \ell=2,\ r=1, \\[6pt]

\displaystyle
-\delta\,
\big\langle \eta_{i,m},\, \bigl((\tfrac12 I + K_\omega')[\eta_{j,n}]\bigr)_i \big\rangle_{\partial D_i},
& \ell=2,\ r=2.
\end{cases}
\end{equation}
for \(i,j=1,\dots,N\), \(|m|,|n|\le F\), where \(\ell=1,2\) labels the
row block corresponding to the first and second equations in
\eqref{eq:galerkin_eqs}, and \(r=1,2\) labels the column block
corresponding to the coefficients of \(\widehat\phi\) and \(\widehat\psi\),
respectively.
The following result follows from classical literature on boundary 
element methods, such as~\cite{colton_kress,saranen2002periodic} in the case where the boundary of the 
resonators is smooth.

\begin{proposition}[Exponential convergence for smooth boundaries]
\label{prop:smooth_convergence}
Let $\partial D_j$ be smooth for each
$j = 1,\dots,N$, and let $\omega$ be a simple resonant frequency of $A_F$ with corresponding eigenmode $u \in L^2(\partial D)$.
Then the following hold:
\begin{itemize}
    \item 
    Let $\omega_F$ denote the resonant frequency computed by the Fourier-Galerkin method with truncation parameter $F$. There exist constants $c > 0$ and $F_0 \in \mathbb{N}$, depending on $\omega$ and the geometry, such that for all $F \geq F_0$,
    \begin{equation}
    \label{eq:exp_conv_omega}
        |\omega_F - \omega| \leq C\,e^{-cF},
    \end{equation}
    where $C > 0$ is independent of $F$.

    \item
    Let $u_F^F \in \mathcal{X}_F$ be the corresponding Galerkin eigenmode, normalised so that $\|u_F^F\|_{L^2(\partial D)} = 1$.
    Then
    \begin{equation}
    \label{eq:exp_conv_mode}
        \|u_F^F - u\|_{L^2(\partial D)} \leq C'\,e^{-c'F},
    \end{equation}
    for constants $C', c' > 0$ independent of $F$.
\end{itemize}
\end{proposition}

\begin{remark}[Non-smooth boundaries]
The Fourier--Galerkin approximation and its exponential convergence rely on the smoothness of $\partial D_j$. For boundaries with corners or edges, the layer potential densities develop singularities and trigonometric polynomials are no longer an efficient approximation space. In that setting, one would instead employ graded-mesh or $hp$-boundary element methods that adapt the approximation space to the corner singularities~\cite{saranen2002periodic}, recovering exponential convergence at the cost of a more involved basis construction. The development of such methods for the subwavelength resonance problem is left for future work.
\end{remark}

\begin{figure}
    \centering
    \begin{subfigure}[b]{0.48\linewidth}
        \centering
        \includegraphics[width=\linewidth]{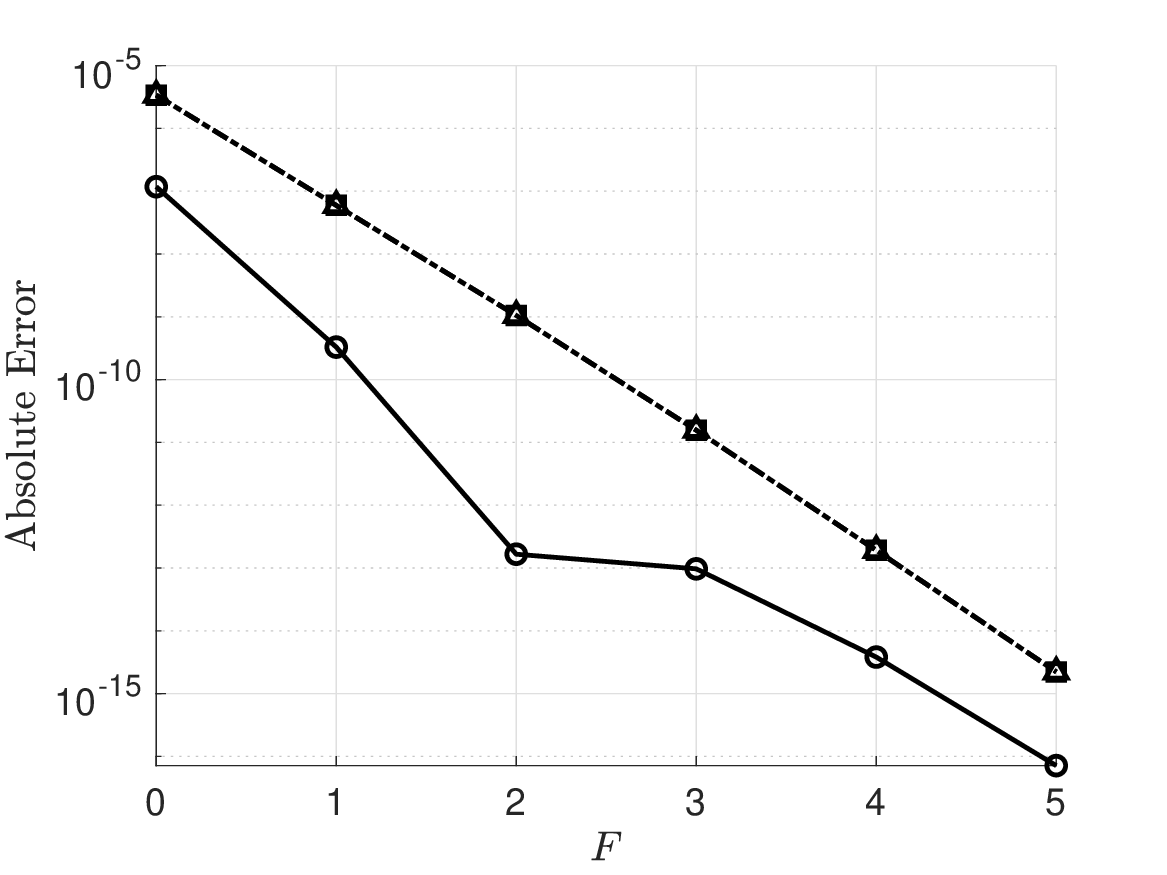}
        \caption{Resonant frequency error $|\omega_F - \omega_{F_{\mathrm{ref}}}|$.}
    \end{subfigure}
    \hfill
    \begin{subfigure}[b]{0.48\linewidth}
        \centering
        \includegraphics[width=\linewidth]{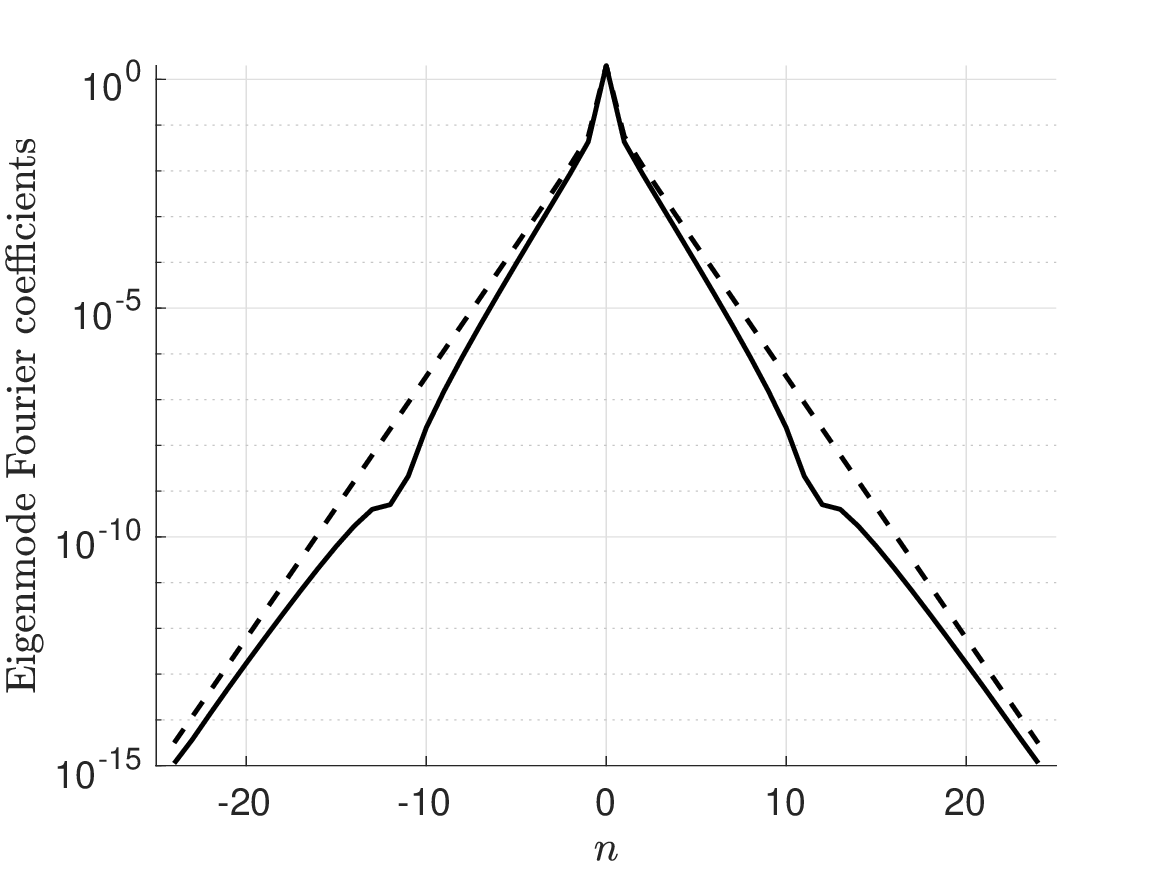}
        \caption{Eigenmode Fourier coefficient decay.}
    \end{subfigure}
    \caption{Exponential convergence in the truncation parameter $F$ for $N = 2$
    circular resonators, confirming Proposition~\ref{prop:smooth_convergence}. The reference solution $(\omega_{\mathrm{ref}},u_{\mathrm{ref}})$ is computed at $F_{\mathrm{ref}} = 6$. Left: resonant frequency error $\abs{\omega_F - \omega_\mathrm{ref}}$. Right: $L^2(\partial D)$ eigenmode error
    $\abs{u_F -u_{\mathrm{ref}}}$, with eigenmodes normalised to unit $L^2(\partial D)$ norm.}
    \label{fig:convergence_in_F}
\end{figure}

\section{Galerkin matrix asymptotic analysis}
\label{sec:asymp}

Despite exponential convergence in the truncation parameter $F$, assembling the full Fourier--Galerkin matrix becomes numerically and computationally prohibitive as the number of resonators increases. Large configurations suffer from ill-conditioning in resonance capture, and non-circular geometries prevent symmetry exploitation. 

To overcome these bottlenecks, we derive two reduced formulations via asymptotic analysis of~\eqref{eq:matrixA} in the limit $\delta \to 0$. These reductions preserve spectral convergence in $F$ while reducing problem dimension and computational cost, and we show how the reduced model yields both accurate initial guesses for resonance refinement and an analytically tractable framework for classifying subwavelength branches.

In Theorem~\ref{thm:reduction}, we present the first effective operator formulation with leading-order error $O(\delta\omega^2\log\omega + \omega^4\log^2\omega)$. The resonance branches divide into two asymptotic regimes: the logarithmic branch with $\omega_{\log} = O(\sqrt{\delta/|\log\delta|})$ and approximation error $O((\delta/|\log\delta|)^{3/2})$, and regular branches with $\omega_j = O(\sqrt{\delta})$ and approximation error $O(\delta^{3/2}|\log\delta|^2)$. This operator admits discretization via truncation parameter $F$ and exhibits spectral convergence in $F$.

Setting $F = 0$ yields the effective matrix formulation of Proposition~\ref{prop:effective_F0}, which reduces the resonance problem to an $N \times N$ eigenvalue problem. Theorem~\ref{thm:branch_scaling} then establishes the precise branch classification: exactly one logarithmically-scaled branch and $N-1$ regular branches exist, with resonance approximation errors scaling as $O(\sqrt{\delta}/\log \delta)$ and $O(\sqrt{\delta/\log \delta})$ respectively.

The benefits of the reduced model are two-fold. First, it supplies high-fidelity initial guesses for the full boundary element method (BEM) solver, enabling robust convergence of Newton refinement. Second, it offers an analytically tractable framework for understanding resonance interaction and mode coupling in subwavelength systems, facilitating both theoretical analysis and predictive insight into the structure of the dispersion landscape.

\subsection{Asymptotic analysis of the integral operators}

First, we detail some asymptotic results for the operators involved in~\eqref{eq:matrixA}. We begin by introducing the operators and kernels relevant to the first row of the operator-valued matrix presented in~\eqref{eq:matrixA}.

For $\mathrm{Im}(\omega) \geq 0$, recall the Helmholtz single-layer operator~\eqref{eq:Somega} 
\begin{equation}
\label{eq:defS}
(S_\omega \varphi)(x) = \int_{\partial D} G_\omega(x-y)\,\varphi(y)\,d\sigma(y),
\end{equation}
mapping $H^{-1/2}(\partial D)$ to $H^{1/2}(\partial D)$, where $G_\omega$ is the outgoing Green's function of the Helmholtz equation in $\mathbb{R}^2$. We also define the kernels
\begin{equation}
\label{eq:defGk}
\begin{aligned}
    G^{(0)}(x,y) &= \frac{1}{2\pi}\log|x-y|, 
    G^{(1)}(x,y) = -\frac{|x-y|^2}{8\pi}, \\
    G^{(2)}(x,y) &= G^{(1)}(x,y)\log\frac{|x-y|}{2} - c_{\mathrm{sl}} |x-y|^2.
\end{aligned}
\end{equation}
The associated operators $S^{(k)}: H^{-1/2}(\partial D)\to H^{1/2}(\partial D)$ are defined as boundary integral operators with kernels $G^{(k)}$, namely
\begin{equation}
\label{eq:defSk}
(S^{(k)} \varphi)(x) = \int_{\partial D} G^{(k)}(x,y)\,\varphi(y)\,d\sigma(y), 
\qquad x \in \partial D,\quad k=0,1,2.
\end{equation}
Here, $d\sigma$ denotes the surface measure on $\partial D$. 

The operator $S^{(0)}$ corresponds to the classical Laplace single-layer operator, while $S^{(1)}$ and $S^{(2)}$ are higher-order correction operators arising from the asymptotic expansion of the Helmholtz Green's function. 
Finally, we introduce the rank-one averaging operator
\begin{equation}
\label{eq:defI}
\mathcal{I}[\varphi] = \int_{\partial D} \varphi(y)\, d\sigma(y). 
\end{equation}
and the frequency-dependent scalar
\begin{equation}
\label{eq:deftau}
\tau_\omega = \frac{\log\omega}{2\pi} - \frac{i}{4} + \frac{\gamma-\log 2}{2\pi},
\end{equation}
where $\gamma$ is the Euler-Mascheroni constant. 
The following expansions follow from the small-argument asymptotics of the Hankel function; see, for example, \cite{ammari2003high} and \cite[Sec.~3.1]{colton_kress}.

\begin{lemma}
\label{lem:expGSK_sobolev}
The operator $S_\omega$ defined in~\eqref{eq:defS} admits the following asymptotic expansion in the operator norm 
$\|\cdot\|_{\mathcal{L}(H^{-1/2}(\partial D), H^{1/2}(\partial D))}$ as $\omega \to 0$:
\begin{equation}
\label{eq:expS_sobolev}
S_\omega -  \tau_\omega \mathcal{I} = S^{(0)}  + \omega^2\log\omega\,S^{(1)} + \omega^2 S^{(2)} + O(\omega^4 \log \omega).
\end{equation}
\end{lemma}
\begin{remark}
It is well known that, in general, $S^{(0)}$ is not invertible. For example the constant function is in its kernel for unit circular domain $\mathbb{T}$:
\begin{equation}
S^{(0)}[\textbf{1}](x) = \frac{1}{2\pi} \int_{\mathbb{T}} \log|x-y| \,d\sigma = \log(1) = 0, \quad \forall x \in  \mathbb{T}.
\end{equation}
    On the other hand, $S^{(0)}$ is coercive, and hence invertible, on the subspace $H^{-1/2}_0(\partial D):=\{\varphi\in H^{-1/2}(\partial D):\int_{\partial D}\varphi\,d\sigma=0\}$ of densities with zero mean \cite{mclean2000strongly}. Consequently, a rank-one modification of $S^{(0)}$ that acts on the mean, such as $S^{(0)}+c\,\mathcal I$ with $|c|$ sufficiently large, is invertible on all of $H^{-1/2}(\partial D)$. This is precisely the structure exploited below, since $S_\omega=S^{(0)}+\tau_\omega\mathcal I+O(\omega^2\log\omega)$ with $|\tau_\omega|\to\infty$ as $\omega\to0$.
\end{remark}

However, we have the following isomorphism argument in the subwavelength regime.
\begin{lemma}\label{lem:S_inv}
For sufficiently small $|\omega| > 0$, the single-layer operator 
$S_\omega: H^{-1/2}(\partial D) \to H^{1/2}(\partial D)$ is an isomorphism.
\end{lemma}

\begin{proof}
Decompose $H^{-1/2}(\partial D) = H_\perp(\partial D) \oplus \mathrm{span}\{\mathbf{1}\}$ and $H^{1/2}(\partial D) = H_\perp^{1/2}(\partial D) \oplus \mathrm{span}\{\mathbf{1}\}$, where $H_\perp(\partial D) := \{\varphi \in H^{-1/2}(\partial D) : \langle\varphi,\mathbf{1}\rangle = 0\}$, and write $S_\omega$ in block form using the respective projections:
\begin{equation}
S_\omega = \begin{pmatrix} P_\perp S_\omega|_{H_\perp} & P_\perp S_\omega|_{\mathrm{span}\{\mathbf{1}\}} \\ P_0 S_\omega|_{H_\perp} & P_0 S_\omega|_{\mathrm{span}\{\mathbf{1}\}} \end{pmatrix} =: \begin{pmatrix} S_\omega^{\perp\perp} & S_\omega^{\perp 0} \\ S_\omega^{0\perp} & S_\omega^{00} \end{pmatrix}.
\end{equation}
Using the expansion $S_\omega = \tau_\omega\mathcal{I} + S^{(0)} + O(\omega^2\log\omega)$ from Lemma~\ref{lem:expGSK_sobolev}, and noting that $\mathcal{I}[\varphi] = 0$ for $\varphi\in H_\perp$ while $\mathcal{I}[\mathbf{1}] = |\partial D|\mathbf{1}$, we identify the blocks as $S_\omega^{\perp\perp} = P_\perp S^{(0)}|_{H_\perp} + O(\omega^2\log\omega)$, $S_\omega^{00} = \tau_\omega|\partial D| + P_0 S^{(0)}[\mathbf{1}] + O(\omega^2\log\omega)$, and $S_\omega^{\perp 0}, S_\omega^{0\perp} = O(1)$. 

Since $\tau_\omega\to\infty$ as $\omega\to 0$, $S_\omega^{00}$ is invertible with $(S_\omega^{00})^{-1} = O(1/\tau_\omega)$. Defining the Schur complement $\Sigma_\omega := S_\omega^{\perp\perp} - S_\omega^{\perp 0}(S_\omega^{00})^{-1}S_\omega^{0\perp}$, we have
\begin{equation*}
\Sigma_\omega = P_\perp S^{(0)}|_{H_\perp} + O(\omega^2\log\omega) + O(1/\tau_\omega) \to P_\perp S^{(0)}|_{H_\perp}
\end{equation*}
as $\omega\to 0$. Since $P_\perp S^{(0)}|_{H_\perp}$ is an isomorphism on $H_\perp$ by coercivity~\cite{mclean2000strongly}, a Neumann series argument gives invertibility of $\Sigma_\omega$ for sufficiently small $\omega$. The block factorisation
\begin{equation}
S_\omega = \begin{pmatrix} I & S_\omega^{\perp 0}(S_\omega^{00})^{-1} \\ 0 & I \end{pmatrix}\begin{pmatrix} \Sigma_\omega & 0 \\ S_\omega^{0\perp} & S_\omega^{00} \end{pmatrix}
\end{equation}
then shows that $S_\omega$ is an isomorphism if and only if both $\Sigma_\omega$ and $S_\omega^{00}$ are invertible, which holds for all sufficiently small $|\omega| > 0$.
\end{proof}

Lemma~\ref{lem:S_inv} establishes the invertibility of $S_\omega$ for small $|\omega|$, which reduces the transmission system~\eqref{eq:block_system} to the single equation $\phi = \psi$ which will be shown in Theorem~\ref{thm:reduction}. It remains to analyze the second row of~\eqref{eq:matrixA}, which governs the reduced operator. To this end, we recall the asymptotic expansions of $K_\omega'$ as $\omega \to 0$. The operator $K_\omega' : L^2(\partial D) \to L^2(\partial D)$ is defined by
\begin{equation*}
(K_\omega'[\varphi])(x) = \int_{\partial D}\partial_{\nu_x}G_\omega(x,y)\varphi(y)\,d\sigma(y), \qquad x\in\partial D,
\end{equation*}
where $G_\omega$ is the outgoing Green's function for the Helmholtz equation. Define operators $K^{(0)'}$, $K^{(1)'}$, $K^{(2)'}$ whose kernels are
\begin{equation}
\label{eq:kernels}
\begin{split}
& \partial_{\nu_x}G^{(0)}(x,y) = \frac{1}{2\pi}\frac{(x-y)\cdot\nu_x}{|x-y|^2},
\qquad
\partial_{\nu_x}G^{(1)}(x,y) = -\frac{(x-y)\cdot\nu_x}{4\pi},
\\ 
& \partial_{\nu_x}G^{(2)}(x,y)
= -\frac{(x-y)\cdot\nu_x}{4\pi}\log\frac{|x-y|}{2}
- c_\gamma(x-y)\cdot\nu_x,\ \textrm{with}\ c_\gamma = \frac{\gamma-\tfrac{1}{2}}{4\pi} - \frac{i}{8}.
\end{split}
\end{equation}
The low-frequency expansion of $K_\omega'$ from the Hankel asymptotics reads (see~\cite{colton_kress, ammari2024functional})
\begin{equation}
\label{eq:expK}
K_\omega' = K^{(0)'}
+ \omega^2\log\omega\,K^{(1)'} + \omega^2 K^{(2)'}
+ O\!\bigl(\omega^4(\log\omega)^2\bigr).
\end{equation}
\begin{theorem}\label{thm:reduction}
For sufficiently small $|\omega| > 0$, the resonances of~\eqref{eq:matrixA} are 
characterised by $\ker(\mathcal{R}(\omega,\delta)) \neq \{0\}$, where the reduced 
operator $\mathcal{R}(\omega, \delta)$ admits the asymptotic expansion:
\begin{equation}
\label{eq:R_expansion}
\mathcal{R}(\omega, \delta) = -\tfrac{1+\delta}{2}I + (1-\delta)K^{(0)'}
+ \omega^2 \log \omega \, K^{(1)'} + \omega^2 K^{(2)'}
+ O(\delta \omega^2 \log \omega + \omega^4 \log^2 \omega).
\end{equation}
The resonance branches divide into two asymptotic regimes:
\begin{enumerate}
    \item Logarithmic branches with $\omega_{\log} = O(\sqrt{\delta/|\log\delta|})$, arising from the dominant balance $\omega^2|\log\omega|\sim\delta$, with approximation error $O\!\left((\delta/|\log\delta|)^{3/2}\right)$.
    \item Regular branches with $\omega_j = O(\sqrt{\delta})$, arising from solutions in $\ker K^{(1)'}$ where the dominant balance $\omega^2 K^{(2)'}\sim\delta$ governs, with approximation error $O(\delta^{3/2}|\log\delta|^2)$.
\end{enumerate}
\end{theorem}
\begin{proof}
The first row of $\mathcal{A}(\omega,\delta)(\phi,\psi)^\top = 0$ yields $S_\omega(\phi-\psi) = 0$. By Lemma~\ref{lem:S_inv}, $S_\omega$ is invertible for sufficiently small $|\omega|>0$, hence $\phi = \psi$. Substituting into the second row gives
\begin{equation}
\label{eq:resonance_K}
\left(-\tfrac{1}{2}I + K_\omega'\right)\phi - \delta\left(\tfrac{1}{2}I + K_\omega'\right)\phi = 0,
\end{equation}
which we write as $\mathcal{R}(\omega,\delta)\phi = 0$ with $\mathcal{R}(\omega,\delta) = -\tfrac{1+\delta}{2}I + (1-\delta)K_\omega'$. Conversely, if $\mathcal{R}(\omega,\delta)\phi = 0$ then setting $\psi = \phi$ satisfies both rows of $\mathcal{A}(\omega,\delta)(\phi,\psi)^\top = 0$, so the resonance condition $\ker\mathcal{A}(\omega,\delta)\neq\{0\}$ is equivalent to $\ker\mathcal{R}(\omega,\delta)\neq\{0\}$. Substituting the expansion~\eqref{eq:expK} and writing $(1-\delta)K_\omega' = K_\omega' - \delta K_\omega'$ gives
\begin{equation}
\begin{aligned}
\mathcal{R}(\omega,\delta) &= -\tfrac{1+\delta}{2}I + (1-\delta)\left(K^{(0)'} + \omega^2\log\omega\,K^{(1)'} + \omega^2 K^{(2)'} + O(\omega^4\log^2\omega)\right) \\
&= -\tfrac{1+\delta}{2}I + (1-\delta)K^{(0)'} + \omega^2\log\omega\,K^{(1)'} + \omega^2 K^{(2)'} + O(\delta\omega^2\log\omega + \omega^4\log^2\omega),
\end{aligned}
\end{equation}
where the $\delta\omega^2$ term is absorbed into $O(\delta\omega^2\log\omega)$ since $\delta\omega^2 = o(\delta\omega^2|\log\omega|)$ as $\omega\to 0$.

For the logarithmic branches, the dominant balance $\omega^2|\log\omega|\sim\delta$ gives $\omega_{\log} = O(\sqrt{\delta/|\log\delta|})$. At this scaling the residual of~\eqref{eq:R_expansion} satisfies $\delta\omega^2|\log\omega| + \omega^4\log^2\omega = O(\delta^2/|\log\delta|)$. Since the resonance condition $\mathcal{R}(\omega,\delta)\phi = 0$ is perturbed by this amount, and the derivative $\partial_{\omega^2}\mathcal{R}\big|_{\omega^2 = \delta/|\log\delta|} = O(|\log\delta|)$ is bounded away from zero, the implicit function theorem gives a true resonance $\omega_{\log}^*$ satisfying
\begin{equation*}
(\omega_{\log}^*)^2 = \omega_{\log}^2 + O\!\left(\frac{\delta^2/|\log\delta|}{|\log\delta|}\right) 
= \omega_{\log}^2 + O\!\left(\frac{\delta^2}{|\log\delta|^2}\right).
\end{equation*}
Converting from error in $\omega^2$ to error in $\omega$, we have 
$$|\omega^* - \omega| = |(\omega^*)^2 - \omega^2|/|\omega^* + \omega| 
= O(\delta^2/|\log\delta|^2) / O(\sqrt{\delta/|\log\delta|}) 
= O\!\left((\delta/|\log\delta|)^{3/2}\right).$$

For the regular branches, we show first $\ker K^{(1)'}$ is nontrivial and infinite-dimensional. Indeed, the kernel of $K^{(1)'}$ contains all $\varphi\in L^2(\partial D)$ satisfying $\int_{\partial D}\varphi\,d\sigma = 0$ and $\int_{\partial D}y\,\varphi(y)\,d\sigma(y) = 0$, 
as follows from the identity
\begin{equation}
(K^{(1)'}[\varphi])(x) = -\frac{\nu_x\cdot x}{4\pi}\int_{\partial D}\varphi\,d\sigma 
+ \frac{\nu_x}{4\pi}\cdot\int_{\partial D}y\,\varphi(y)\,d\sigma(y),
\end{equation}
For solutions $\phi\in\ker K^{(1)'}$, the logarithmic term vanishes and the dominant balance becomes $\omega^2 K^{(2)'}\phi\sim\delta\phi$, giving $\omega_j = O(\sqrt{\delta})$. At this scaling the error in $\mathcal{R}$ is $\delta\omega^2|\log\omega| + \omega^4\log^2\omega = O(\delta^2|\log\delta|^2)$, and $\partial_{\omega^2}\mathcal{R} = O(1)$, so the error in $\omega^2$ is $O(\delta^2|\log\delta|^2)$. Dividing by $\omega_j = O(\sqrt{\delta})$ gives approximation error $O(\delta^{3/2}|\log\delta|^2)$ on $\omega_j$.
\end{proof}

\subsection{Effective matrix for subwavelength resonances}
While Theorem~\ref{thm:reduction} establishes a highly accurate reduced operator equation, applying a Galerkin discretization directly to this formulation leaves two practical challenges: the multiplicity of the resonance branches is not explicit, and the underlying nonlinear eigenvalue problem requires a robust initial guess for numerical root-finding. In this section, we derive an effective eigenvalue equation (Theorem~\ref{thm:branch_scaling}) that resolves both issues. This formulation specifies the branch counts and provides high-fidelity asymptotic approximations, yielding initial guesses accurate to $O(\sqrt{\delta}/|\log \delta|)$ and $O(\sqrt{\delta/|\log \delta|})$ for the logarithmic and regular branches, respectively.

Recall the Fourier-Galerkin space $\mathcal X_F$ defined in Section~\ref{sec:BEM} in terms of the parametrization $x_j(\theta)$, $\theta \in [0,2\pi)$, on each $\partial D_j$. We equip it with an orthonormal basis by normalizing the trigonometric basis:
\begin{equation}
\label{eq:basis_decomp}
\mathcal X_F=\mathcal X_{F,0}\oplus \mathcal X_{F,\perp},
\end{equation}
where the orthonormal basis functions are
\begin{equation}
\label{eq:basis}
\eta_{j,n}(\theta) = \frac{1}{\sqrt{|\partial D_j|}} e^{in\theta},
\qquad j=1,\dots,N,\quad |n|\le F,
\end{equation}
The subspaces are
\begin{equation}
\mathcal X_{F,0}:=\mathrm{span}\{\eta_{j,0}:j=1,\dots,N\},
\qquad
\mathcal X_{F,\perp}:=\mathrm{span}\{\eta_{j,n}:j=1,\dots,N,\ 0<|n|\le F\}.
\end{equation}
For $F \geq 0$, the three Galerkin matrices of $-\tfrac{1}{2}I + K^{(0)'}$, $K^{(1)'}$, $K^{(2)'}$ are given by
\begin{equation}
\label{eq:galerkin_matrices}
\begin{split}
   & (C_0^F)_{(i,m),(j,n)} = \langle e^{im\theta},\,(-\tfrac{1}{2}I+K^{(0)'})[e^{in\cdot}\chi_{\partial D_j}]\rangle_{\partial D_i},
   \\
& (K_1^F)_{(i,m),(j,n)} = \langle e^{im\theta},\,K^{(1)'}[e^{in\cdot}\chi_{\partial D_j}]\rangle_{\partial D_i},
\\
& (K_2^F)_{(i,m),(j,n)} = \langle e^{im\theta},\,K^{(2)'}[e^{in\cdot}\chi_{\partial D_j}]\rangle_{\partial D_i}.
\end{split}
\end{equation}
The Galerkin matrix of $\mathcal{R}$ is then given by
\begin{equation}
\label{eq:RF}
R_F(\omega,\delta) = (1-\delta)C_0^F - \delta I_{N(2F+1)}
+ \omega^2\log\omega\,K_1^F + \omega^2 K_2^F
+ O(\delta\omega^2\log\omega,\,\omega^4(\log\omega)^2).
\end{equation}
The resonance condition is $\det R_F(\omega,\delta) = 0$,
a nonlinear eigenvalue problem of size $N(2F+1)$, solved for each branch by finding a null vector $\mathbf{c}^\ell \in \mathbb{C}^{N(2F+1)}$ satisfying $R_F(\omega_\ell,\delta)\,\mathbf{c}^\ell = 0$. The components of $\mathbf{c}^\ell$ are indexed by pairs $(j,n)$ with $j=1,\dots,N$ and $|n|\leq F$, and the corresponding approximate eigenmode is
\begin{equation}
\label{eq:eigenmode_approx}
u_\ell^F(x)
= \sum_{j=1}^N \sum_{|n|\leq F} c_{j,n}^{(\ell)}\,\eta_{j,n}(x)
= \sum_{j=1}^N \frac{\chi_{\partial D_j}(x)}{\sqrt{|\partial D_j|}}
  \sum_{|n|\leq F} c_{j,n}^{(\ell)}\,e^{in\theta_j(x)},
\end{equation}
where $\theta_j(x)$ denotes the parametrisation angle of the point
$x\in\partial D_j$.
Thus $u_\ell^F$ is supported on $\bigcup_j \partial D_j$ and,
on each boundary $\partial D_j$, is a trigonometric polynomial of degree $F$
with coefficients $\bigl(c_{j,n}^{(\ell)}\bigr)_{|n|\leq F}$.

The following Proposition presents the explicit reduction of the $F=0$ 
truncation of $R_F(\omega,\delta)$ to an $N\times N$ effective eigenvalue 
problem, in which the resonance condition decouples from the oscillatory 
modes $\mathcal{X}_{F,\perp}$ and is governed entirely by the constant-mode 
subspace $\mathcal{X}_{F,0}$.

\begin{proposition}
\label{prop:effective_F0}
For $F = 0$, the Galerkin resonance condition $\det R_F(\omega,\delta) = 0$ in $\mathcal{X}_{F,0} = \operatorname{span}\{\eta_{j,0} : j=1,\dots,N\}$ is equivalent to
\begin{equation}
\label{eq:effective_eq}
C_{\operatorname{eff}}(\omega,\delta)\,x = O(\omega^4(\log\omega)^2),
\end{equation}
where the effective capacitance matrix is defined as
\begin{equation}
C_{\operatorname{eff}}(\omega,\delta) := \delta I_N - \omega^2\log\omega\,K_1^0 - \omega^2 K_2^0,
\end{equation}
and the matrices $K_1^0, K_2^0 \in \mathbb{C}^{N\times N}$ have entries
\begin{equation}
\label{eq:H1}
(K_1^0)_{ij} = \langle \eta_{i,0}, K^{(1)'} \eta_{j,0} \rangle = -\frac{|D_i|}{2\pi}\sqrt{\frac{|\partial D_j|}{|\partial D_i|}},
\end{equation}
\begin{equation}
\label{eq:H2}
\begin{split}
    & (K_2^0)_{ij} =  \langle \eta_{i,0}, K^{(2)'} \eta_{j,0} \rangle\\ = & -\frac{1}{4\pi\sqrt{|\partial D_i||\partial D_j|}}
\int_{\partial D_i}\int_{\partial D_j}
\bigl((x-y)\cdot \nu_x\bigr)
\log\!\left(\frac{|x-y|}{2}\right)d\sigma(y)\,d\sigma(x)
+
\left(\gamma-\tfrac12-\frac{i\pi}{2}\right)(K_1^0)_{ij}.
\end{split}
\end{equation}
\end{proposition}

\begin{proof}
The Galerkin resonance condition $\det R_F(\omega,\delta) = 0$ involves the operator 
$(-\frac{1}{2}I + K^{(0)'})$ for $F=0$. For a constant density $\rho=1$ on $\partial D_j$, 
the single-layer potential is 
constant for all $x \in D_j$, so its gradient vanishes inside $D_j$ and the internal 
normal derivative is zero: $(-\frac{1}{2}I + K^{(0)'})[1] = 0$. Thus 
$\langle \eta_{i,0}, (-\frac{1}{2}I + K^{(0)'})\eta_{j,0}\rangle = 0$, and the resonance 
condition reduces to $\det(\delta I_N - \omega^2\log\omega\,K_1^0 - \omega^2 K_2^0) = 0$.

Using $\eta_{j,0} = |\partial D_j|^{-1/2}$, the entry $(K_1^0)_{ij}$ follows from the 
divergence theorem:
\begin{equation*}
(K_1^0)_{ij} 
= \frac{1}{\sqrt{|\partial D_i||\partial D_j|}} \int_{\partial D_i} \int_{\partial D_j} 
\frac{-(x-y)\cdot\nu_x}{4\pi}\,d\sigma(y)\,d\sigma(x) 
= -\frac{|D_i|}{2\pi}\sqrt{\frac{|\partial D_j|}{|\partial D_i|}},
\end{equation*}
where we used $\int_{\partial D_j} d\sigma(y) = |\partial D_j|$ and 
$\int_{\partial D_i}(x-y)\cdot\nu_x\,d\sigma(x) = 2|D_i|$ by the divergence theorem.

For $(K_2^0)_{ij}$, recall that the kernel of $K^{(2)'}$ is
\begin{equation*}
K^{(2)'}(x,y) = -\frac{(x-y)\cdot\nu_x}{4\pi}\log\frac{|x-y|}{2} 
- \left(\gamma - \tfrac{1}{2} - \tfrac{i\pi}{2}\right)\frac{(x-y)\cdot\nu_x}{4\pi},
\end{equation*}
where the second term arises from the expansion of the Hankel function $H_0^{(1)}$ at 
small argument. Substituting into $\langle \eta_{i,0}, K^{(2)'}\eta_{j,0}\rangle$ and 
splitting the two contributions gives
\begin{equation*}
(K_2^0)_{ij} 
= -\frac{1}{\sqrt{|\partial D_i||\partial D_j|}} \int_{\partial D_i}\int_{\partial D_j}
\frac{(x-y)\cdot\nu_x}{4\pi}\log\frac{|x-y|}{2}\,d\sigma(y)\,d\sigma(x)
+ \left(\gamma - \tfrac{1}{2} - \tfrac{i\pi}{2}\right)(K_1^0)_{ij},
\end{equation*}
where the second term uses the same divergence theorem computation as for $(K_1^0)_{ij}$. 
Dividing by $\sqrt{|\partial D_i||\partial D_j|}$ yields the stated formula \eqref{eq:H2}.
\end{proof}

The eigenvalue problem in \eqref{eq:effective_eq} can be further analyzed by exploiting the spectral structure of $K_1^0$, which facilitates a classification of the resonance branches. A direct computation shows that $K_1^0$ is a rank-one matrix, admitting the outer product representation $K_1^0 =-\frac{1}{2\pi} \, v_1 w_1^\top$, where
\begin{equation}
\label{eq:H1_rank1}
v_1 = \left(\frac{|D_1|}{\sqrt{|\partial D_1|}},\dots,\frac{|D_N|}{\sqrt{|\partial D_N|}}\right)^{\!\top},
\quad
w_1 = \left(\sqrt{|\partial D_1|},\dots,\sqrt{|\partial D_N|}\right)^{\!\top}.
\end{equation}
Consequently, the kernel of $K_1^0$ is given by
\begin{equation}
\label{eq:K10}
\ker K_1^0 = \left\{x\in\mathbb{C}^N : \sum_{j=1}^N \sqrt{|\partial D_j|}\,x_j=0\right\}.
\end{equation}
The matrix $K_1^0$ has one non-zero eigenvalue $\mu_1 = -\frac{1}{2\pi} \sum_{i = 1}^N \abs{D_i}<0$ corresponding to the eigenvector $v_1$. This rank-one structure allows the eigenvalue problem \eqref{eq:effective_eq} to decouple into a scalar equation in $V = \mathrm{span}\{v_1\}$ and an $(N-1)\times(N-1)$ reduced system on $V^\perp = \ker K_1^0$. We express $K_2^0$ in block form with respect to the decomposition $\mathbb{C}^N = V \oplus V^\perp$:
\begin{equation}
\label{eq:K20}
K_2^0 =
\begin{pmatrix}
\alpha & b^* \\
c & B
\end{pmatrix},
\qquad
\alpha = P_V K_2^0 P_V \in \mathbb{C},
\quad
B = P_{V^\perp} K_2^0 P_{V^\perp} \in \mathbb{C}^{(N-1)\times(N-1)},
\end{equation}
where $P_V$ and $P_{V^\perp}$ denote the orthogonal projections onto $V$ and $V^\perp$, respectively. The following theorem characterizes the $N$ branches of subwavelength resonances.

\begin{theorem}
\label{thm:branch_scaling}
Assume the matrix $B$ is invertible. For sufficiently small $\delta > 0$, there exist one logarithmic branch and $N-1$ regular branches for the eigenvalue problem~\eqref{eq:effective_eq}. More precisely,
\begin{equation}
\label{eq:log_branch}
\omega_{\log}^2
= \frac{2\delta}{\mu_1 \log\delta - \mu_1 \log(\frac{1}{2}\mu_1 \log\delta) + 2\alpha}
+ O\!\left(\frac{\delta}{|\log\delta|^2}\right),
\end{equation}
and, for $j = 2,\dots,N$,
\begin{equation}
\label{eq:reg_branch}
\omega_j^2 = \frac{\delta}{\nu_j}
+ O\!\left(\frac{\delta}{|\log\delta|}\right),
\end{equation}
where $\nu_2,\dots,\nu_N$ are the eigenvalues of $B$.
\end{theorem}
\begin{proof}
Since $K_1^0 = -\frac{1}{2\pi}v_1 w_1^\top$ is rank one with range $V = \mathrm{span}\{v_1\}$ and vanishes identically on $V^\perp = \ker K_1^0$, the matrix $C_{\operatorname{eff}}$ takes the block form in the basis $V\oplus V^\perp$:
\begin{equation}
\label{eq:block}
C_{\operatorname{eff}}(\omega,\delta) = \begin{pmatrix}
\delta - \omega^2\log\omega\,\mu_1 - \omega^2\alpha & -\omega^2 b^* \\
-\omega^2 c & \delta I_{N-1} - \omega^2 B
\end{pmatrix},
\end{equation}
where the $\log\omega$ term appears only in the $(V,V)$ entry. We analyse the resonance condition $\det C_{\operatorname{eff}} = 0$ in two asymptotic regimes.

First, at the scaling $\omega^2 = O(\delta)$, we expand $\det C_{\operatorname{eff}}$ along the first row to obtain
\begin{equation}
\label{eq:det_expand}
\det C_{\operatorname{eff}} = (\delta - \omega^2\log\omega\,\mu_1 - \omega^2\alpha)
\det(\delta I_{N-1} - \omega^2 B) + O(\delta^N).
\end{equation}
Near each root $\omega^2 = \delta/\nu_j$, the factor $(\delta - \omega^2\log\omega\,\mu_1 - \omega^2\alpha) = O(\delta)$ is bounded away from zero, so the resonance condition~\eqref{eq:effective_eq} with \eqref{eq:det_expand} reduces at leading order to $\det(\delta I_{N-1} - \omega^2 B) = 0$, whose solutions are $\omega^2 = \delta/\nu_j$ for $j = 2,\dots,N$. An application of Rouch\'e's theorem, comparing $\det C_{\operatorname{eff}}$ to $\det(\delta I_{N-1} - \omega^2 B)\cdot(\delta - \omega^2\log\omega\,\mu_1 - \omega^2\alpha)$ on small contours enclosing each $\delta/\nu_j$ in the $\omega^2$-plane, confirms that exactly $N-1$ resonances exist at this scaling. Therefore we have
\begin{equation}
\omega_j^2 = \frac{\delta}{\nu_j} + O\!\left(\frac{\delta}{|\log\delta|}\right), 
\qquad j = 2,\dots,N,
\end{equation}

Secondly, at the scaling $\omega^2 = O(\delta/|\log\delta|)$, we have $\omega^2 \ll \delta$, so $\delta I_{N-1} - \omega^2 B$ is invertible for small $\delta$. The Schur complement of $C_{\operatorname{eff}}$ onto the $V$ block is therefore well-defined, and the resonance condition $\det C_{\operatorname{eff}} = 0$ is equivalent to
\begin{equation}
\label{eq:schur}
\delta - \omega^2(\mu_1\log\omega + \alpha) - \omega^4 c^*(\delta I_{N-1} - \omega^2 B)^{-1}b = 0,
\end{equation}
with the Schur complement term satisfying $\omega^4 c^*(\delta I_{N-1} - \omega^2 B)^{-1}b = O(\delta/|\log\delta|^2)$, and \eqref{eq:schur} reduces at leading order to the scalar equation $\delta - \omega^2\mu_1\log\omega = 0$. Setting $\omega = e^{\tilde\omega}$ with $\tilde\omega \to -\infty$, the leading-order equation becomes $|\mu_1||\tilde\omega|e^{2\tilde\omega} \sim \delta$, since $\mu_1<0$ and $\tilde\omega<0$. An application of Rouch\'e's theorem on a contour in the $\tilde\omega$-plane enclosing the region 
where this balance holds confirms that exactly one resonance exists at this scaling. 
To compute the asymptotics of the logarithmic branch, we apply the logarithmic-substitution method to $|\mu_1||\tilde\omega|e^{2\tilde\omega} \sim \delta$. Taking logarithms gives
\begin{equation}
\label{eq:logbalance}
2\tilde\omega + \log(|\mu_1||\tilde\omega|) = \log\delta + O\!\left(\tfrac{1}{|\log\delta|}\right).
\end{equation}
The linear term dominates, giving leading order $\tilde\omega \sim \tfrac{1}{2}\log\delta$. Writing $2\tilde\omega = \log\delta + y$ with $y = o(\log\delta)$ and substituting into \eqref{eq:logbalance} yields $y = -\log(\tfrac{1}{2}\mu_1\log\delta) + O(\log|\log\delta|/|\log\delta|)$, and hence
\begin{equation}
\label{eq:w_expansion}
\tilde\omega = \tfrac{1}{2}\log\delta - \tfrac{1}{2}\log\!\left(\tfrac{1}{2}\mu_1\log\delta\right) + O\!\left(\frac{\log|\log\delta|}{|\log\delta|}\right).
\end{equation}
Substituting \eqref{eq:w_expansion} into the denominator of \eqref{eq:schur} and recalling $\omega_{\log}^2 = e^{2\tilde\omega}$, we obtain
\begin{equation}
\omega_{\log}^2 = \frac{2\delta}{\mu_1\log\delta - \mu_1\log(\frac{1}{2}\mu_1\log\delta) + 2\alpha} + O\!\left(\frac{\delta}{|\log\delta|^2}\right).
\end{equation}

The two scalings are asymptotically disjoint since $\delta/|\log\delta| \ll \delta$, so together they account for all $N$ branches.
\end{proof}

\begin{figure}[h]
    \centering
    \includegraphics[width=\linewidth]{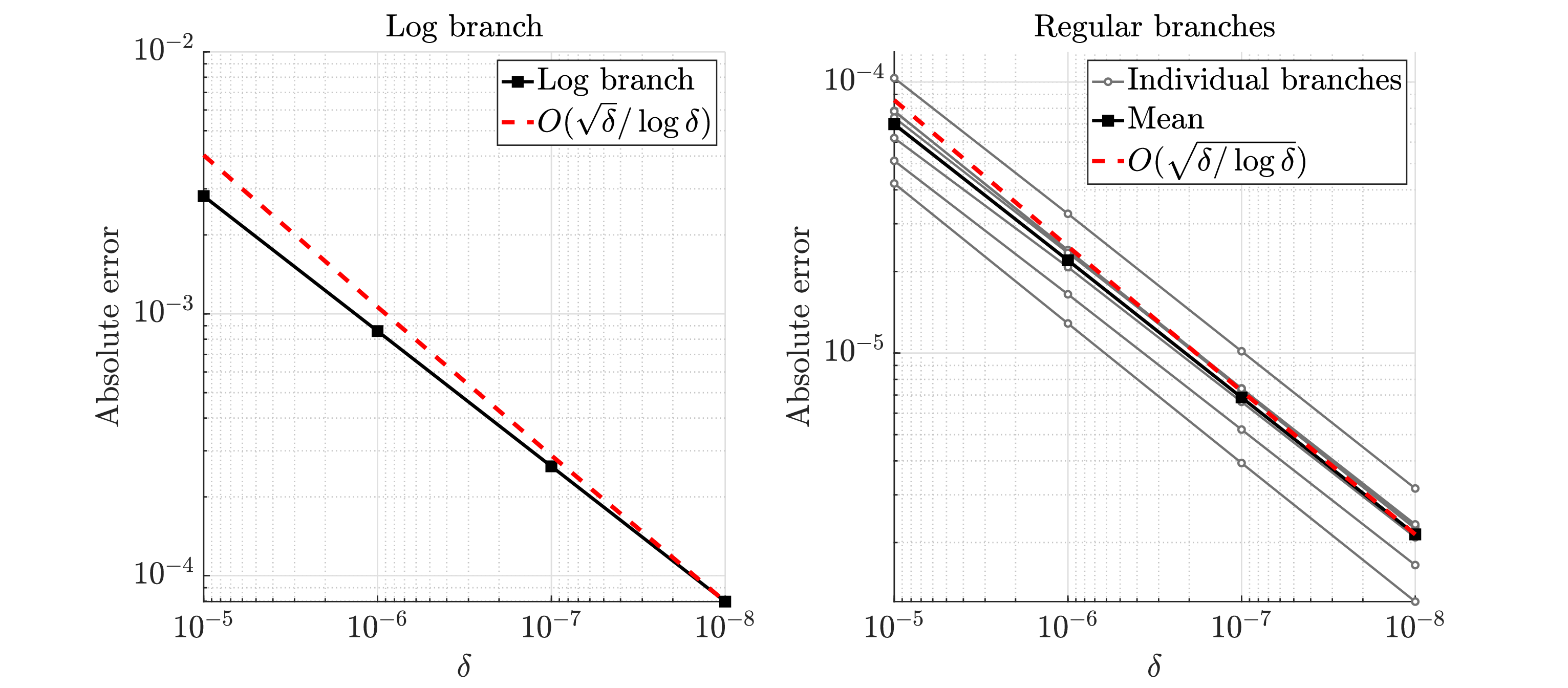}
\caption{Resonance error convergence for $N=8$ circular resonators, validating Theorem~\ref{thm:branch_scaling}.}    \label{fig:branch_convergence}
\end{figure}
\begin{remark}
  Figure \ref{fig:branch_convergence} shows resonance convergence for a numerical simulation with $N = 8$ circular resonators. The left panel demonstrates that the logarithmic branch achieves substantially smaller absolute error than the regular branches across all scales of $\delta$, consistent with Theorem \ref{thm:branch_scaling}'s prediction of superlinear convergence. Both panels confirm the predicted asymptotic scalings in the regime $\delta \lesssim 10^{-6}$. The right panel displays individual errors for each of the $N-1$ regular branches (faint dashed lines), which exhibit expected variability due to the distinct eigenvalue ratios $\nu_j$ of the system; their mean (solid line) tracks the asymptotic envelope, validating the $\delta$-scaling uniformly across branches. 
\end{remark}

\begin{remark}
  Figure \ref{fig:branch_convergence2} shows resonance convergence for $N = 8$ circular resonators across Fourier truncation levels $F = 0, \ldots, 6$ (indicated by the colorbar). The logarithmic branch (left panel) exhibits absolute error that tracks the theoretical rate $O((\delta/|\log\delta|)^{3/2})$, while the regular branches (right panel, mean) follow $O(\delta^{3/2}\abs{\log \delta}^2)$. These error bounds confirm that Theorem~\ref{thm:reduction} characterizes the subwavelength resonance structure.
\end{remark}

\begin{figure}[H]
    \centering
    \includegraphics[width=\linewidth]{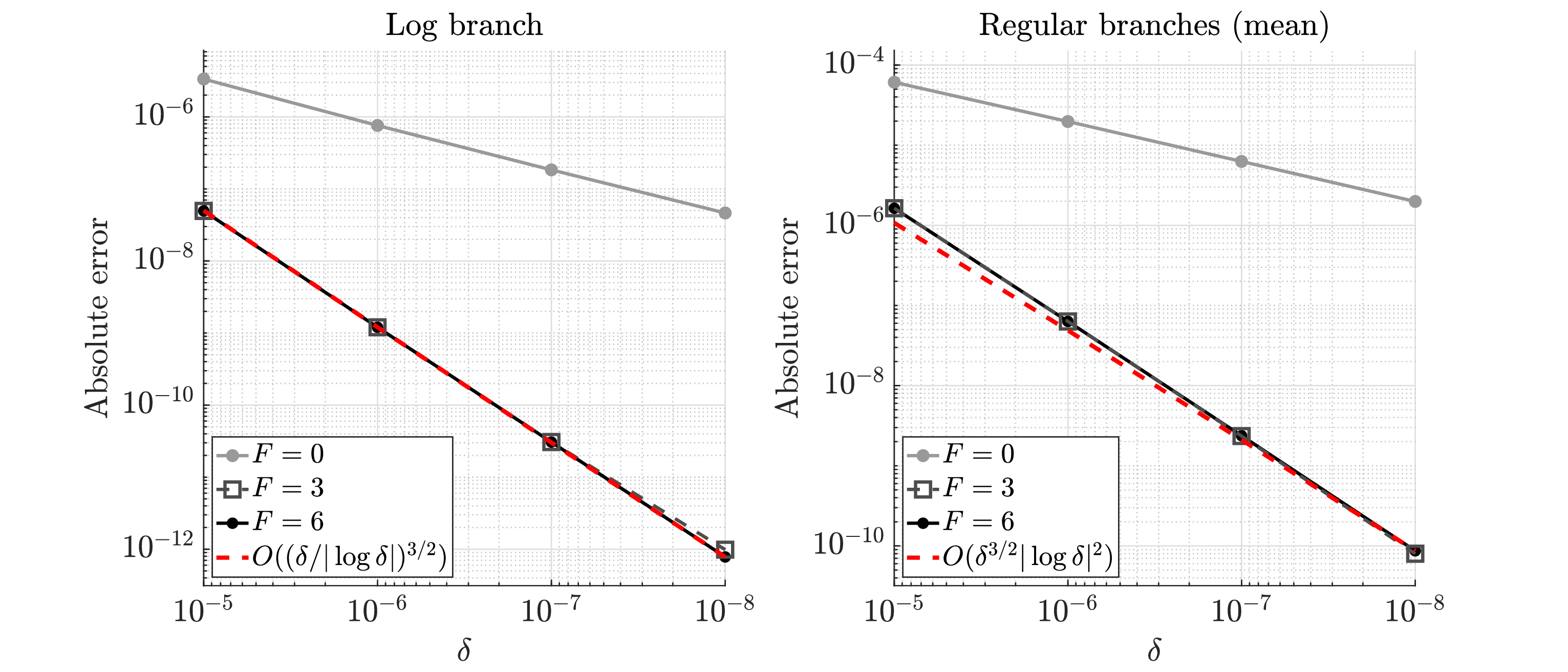}
    \caption{Resonance error convergence for $N=8$ circular resonators, validating Theorem~\ref{thm:reduction}.}
    \label{fig:branch_convergence2}
\end{figure}

\begin{remark}[Relation to the Capacitance Matrix Approach]
Proposition~\ref{prop:effective_F0} provides the $F=0$ analogue of the capacitance matrix of Ammari et al.~\cite{ammari2024functional}. Both reduce the resonance problem to an $N\times N$ system, but they differ significantly in their construction and handling of two-dimensional logarithmic degeneracies.

In the framework of Ammari et al., the $N\times N$ system arises from a Lyapunov--Schmidt reduction onto the null space of $-\tfrac{1}{2}I + K^{(0)'}$, spanned by the constant functions. This approach can be interpreted as a Petrov--Galerkin discretization, where the entries $C_{ij} = \langle \mathbf{1}_{\partial D_i}, (S^{(0)})^{-1}\mathbf{1}_{\partial D_j}\rangle$ effectively involve the inverse of the Laplace single-layer operator $S^{(0)}$. This inversion is ill-posed in two dimensions and necessitates specialized treatment of the logarithmic term.

The effective matrix $C_{\mathrm{eff}}$ introduced in this work avoids the inversion of $S^{(0)}$ entirely. By employing a standard Galerkin projection, the entries of $K_1^0$ and $K_2^0$ are computed via direct boundary integrals of the kernels $\partial_{\nu_x}G^{(1)}$ and $\partial_{\nu_x}G^{(2)}$. This enables a direct kernel analysis as shown in Theorem~\ref{thm:branch_scaling}.
\end{remark}

\section{Fast computation of subwavelength resonances}
\label{sec:num}
A common strategy for identifying resonances in general domains is to perform a grid search in the complex $\omega$-plane to locate points where the Galerkin matrix $A_F(\omega, \delta)$ exhibits small singular values. These candidate points are subsequently refined using local root-finding methods, such as Müller's method or Newton iteration.

However, this brute-force approach is often computationally inefficient and unreliable. The initial grid search necessitates repeated assembly and evaluation of $A_F(\omega, \delta)$ over a two-dimensional parameter space, incurring significant computational cost. Furthermore, when resonances are tightly clustered or exhibit high sensitivity to the material contrast $\delta$, the landscape of the smallest singular value becomes ill-conditioned. This makes it difficult to distinguish true resonances from numerical artifacts and renders local iterative methods highly dependent on the quality of initial guesses, often without any \textit{a priori} guarantee regarding the existence or location of resonances.

In contrast, for structured geometries such as circular inclusions, the asymptotic effective matrix $C_{\operatorname{eff}}$ derived in Theorem~\ref{thm:branch_scaling} allows for explicit calculation of the resonance locations. These asymptotic approximations provide reliable initial guesses, which significantly enhance the efficiency and robustness of numerical solvers and eliminate the need for costly global searches.

We summarize the computation of the resonances as follows. The nonlinear eigenvalue search in the final BEM refinement step is a standard scalar Newton iteration applied to the smallest singular value of $A_F(\omega)$; Müller's method \cite{muller1956} may be substituted when derivative evaluation is expensive.

\begin{algorithm}
\caption{Three-level resonance solver for systems of subwavelength resonators}
\label{alg:resonance_solver}
\begin{algorithmic}[1]
\State \textit{(Asymptotic seeds, Theorem~\ref{thm:branch_scaling}.)} Assemble $K_1^0$ and $K_2^0$ from the resonator geometry. Extract $\mu_1$, $\alpha$, and the eigenvalues $\nu_2,\dots,\nu_N$ of $B = P_{V^\perp} K_2^0 P_{V^\perp}$~\eqref{eq:K20}. Evaluate the asymptotic seeds $\omega_{\log}^{(0)}$ and $\omega_j^{(0)}$, $j=2,\dots,N$, from \eqref{eq:log_branch}--\eqref{eq:reg_branch}, together with the corresponding null vectors $v_j^{(0)}$ embedded in the Galerkin space.
\State \textit{(Galerkin assembly, Theorem~\ref{thm:reduction}.)} Assemble $C_0^F$, $K_1^F$, $K_2^F$ from~\eqref{eq:galerkin_matrices}. Initialize $\omega_0 \gets \operatorname{mean}(|\omega_j^{(0)}|)$.
\Repeat
    \State Solve the generalized eigenvalue problem~\eqref{eq:RF}.
    \State Match eigenpairs to $v_j^{(0)}$ by maximum overlap. Set $\omega_0 \gets \operatorname{mean}(|\omega_j^{\mathrm{eff}}|)$.
\Until{$|\omega_0^{\mathrm{new}} - \omega_0| < \varepsilon\,|\omega_0|$} for a tolerance $\varepsilon$
\For{each effective resonance $\omega_j^{\mathrm{eff}}$, $j = 1,\dots,N$}
    \State \textit{(BEM refinement.)} Assemble $A_F(\omega_j^{\mathrm{eff}})$ and initialize $(u,v)$ from its smallest singular triplet.
    \Repeat
        \State Set $g \gets u^* A_F(\omega)\,v$, $g' \gets u^* A_F'(\omega)\,v$. Update $\omega \gets \omega - g/g'$, with step capped at $|\Delta\omega| \le \tfrac{1}{2}|\omega|$. Refresh $(u,v)$ via SVD of $A_F(\omega)$.
    \Until{$|\Delta\omega| < \varepsilon\,|\omega|$}
    \If{$\sigma_{\min}(A_F(\omega)) < \varepsilon$ \textbf{and} $\min_{k<j}|\omega - \omega_k^{\mathrm{bem}}| > \varepsilon_{\mathrm{dist}}$}
        \State Accept $\omega_j^{\mathrm{bem}} \gets \omega$ and record the null vector $v$.
    \Else
        \State Discard as spurious or duplicate.
    \EndIf
\EndFor
\end{algorithmic}
\end{algorithm}

In Section~\ref{sec:numcircle}, we investigate numerical methods for circular resonators, which constitute the fundamental building blocks of subwavelength metamaterials. Owing to the availability of explicit analytical expressions for the single-layer and double-layer potential operators on circular domains, this geometry serves as an ideal benchmark for assessing the accuracy of the asymptotic formulas developed in this paper.

Subsequently, Section~\ref{sec:num2} details the numerical treatment for general smooth domains, providing examples for complex structures such as ring-shaped and elliptical resonators. Both sections explain numerical treatment for the implementation of the Asymptotic formula~\eqref{eq:RF} from Theorem~\ref{thm:reduction} and the full BEM to solve~\eqref{eq:AN_full_entries} for different domains.

\subsection{Circular resonators}
\label{sec:numcircle}
In the special case of circular inclusions with radius $a_j$ and center $c_j$, we parametrize
\begin{equation}
\partial D_j=\{x_j(\theta)=c_j+a_j(\cos\theta,\sin\theta):\theta\in[0,2\pi)\},
\qquad j=1,\dots,N.
\end{equation}

\subsubsection{FFT based full boundary integral equation}
A naive implementation of the full multipole method for \eqref{eq:AN_full_entries}
leads to a computational cost of $O(F^2N^2)$. This can be significantly
reduced by exploiting the rotational invariance of each boundary. On each $\partial D_j$, we expand the density in a truncated Fourier basis:
\begin{equation}
\phi_j(\theta)=\sum_{n=-F}^{F}\hat{\phi}_{j,n}e^{in\theta}.
\end{equation}
The use of equispaced nodes allows the Fast Fourier Transform (FFT) to evaluate and transform these series in $O(F\log F)$ operations.
We now analyze the structure of the boundary integral operators in this basis by starting with the off-diagonal blocks ($i\neq j$).
Let
\[
c_i-c_j=d_{ij}(\cos\alpha_{ij},\sin\alpha_{ij}), \qquad d_{ij}>a_j.
\]
By Graf's addition theorem (see, e.g., \cite{colton_kress}), the field generated
by inclusion $i$ and restricted to $\partial D_j$ admits the expansion
\begin{equation}
u_i(x_j(\theta))
=
\sum_{n=-F}^{F}
\left(
\sum_{m=-F}^{F}
b_{i,m}\,
H_{m-n}^{(1)}(\omega d_{ij})\,e^{i(m-n)\alpha_{ij}}
\right)
J_n(\omega a_j)e^{in\theta}.
\end{equation}
Thus, in the Fourier basis, the off-diagonal block $A_{ij}$ is given by
\begin{equation}
\label{eq:Aij_modes}
(A_{ij})_{nm}
=
J_n(\omega a_j)\,
H_{m-n}^{(1)}(\omega d_{ij})\,e^{i(m-n)\alpha_{ij}},
\end{equation}
which depends only on the index difference $m-n$. Hence each block is Toeplitz (up to a phase factor), and matrix--vector products can be evaluated via FFT in $O(F\log F)$ operations.
For the diagonal blocks ($i=j$), we consider $x_j(\theta),y_j(\phi)\in\partial D_j$, where
\[
|x_j(\theta)-y_j(\phi)|
=
2a_j\left|\sin\frac{\theta-\phi}{2}\right|,
\]
so the kernel depends only on $\theta-\phi$ and the operator is a convolution. Setting $z_j=\omega a_j$, one has the Fourier expansion
\begin{equation}
H_0^{(1)}\!\bigl(\omega|x_j(\theta)-y_j(\phi)|\bigr)
=
\sum_{m\in\mathbb{Z}}
J_m(z_j)H_m^{(1)}(z_j)e^{im(\theta-\phi)}.
\end{equation}
Therefore, the single-layer operator is diagonal in the Fourier basis:
\begin{equation}
\label{eq:Ajj_modes}
(S_{jj}^\omega)_{mm}
=
-\frac{i\pi a_j}{2}\,J_m(\omega a_j)H_m^{(1)}(\omega a_j).
\end{equation}
Similarly, the adjoint double-layer operator is diagonal with
\begin{equation}
(K'_{\omega,jj})_{mm} = -\frac{i\pi \omega a_j}{4}
\Bigl(
J_m(z_j)H_m^{(1)\,'}(z_j)
+
J_m'(z_j)H_m^{(1)}(z_j)
\Bigr),
\end{equation}
to which the jump terms $\pm \tfrac12$ of the exterior and interior traces are added.
In conclusion, the off-diagonal blocks exhibit a convolution structure in the modal index,
while the diagonal blocks are exactly diagonal. As a result, all block
matrix--vector products can be carried out using FFT-based convolution,
leading to an overall computational cost of
\[
O(F^2N + NF\log F).
\]
\subsubsection{Effective Galerkin matrices assembly}
\label{sec:galerkin_assembly}
Without prior knowledge of the asymptotic behavior of $\omega$ as $\delta\to 0$, a naive approach would require a two-dimensional search in the complex plane, with each iteration involving assembly of the full matrix $A_F$, despite the FFT-based acceleration explained above. For closely spaced resonances, such iterative methods are unstable and provide no \emph{a priori} information on the number of resonances. This motivates the use of the effective equations in Theorems~\ref{thm:reduction} and~\ref{thm:branch_scaling}.

Each of the three $N(2F+1)\times N(2F+1)$ block-structured matrices $C_0^F$, $K_1^F$, $K_2^F$ from~\ref{eq:galerkin_matrices} are assembled resonator by resonator. For circles, the diagonal block of resonator $i$ is circulant, and its Fourier symbol is available in closed form. For $C_0^F$ the symbol is $-\tfrac{1}{2}$ for $m\neq 0$ and $0$ for $m=0$, reflecting the eigenvalues of $K^{(0)'}$ on a circle. For $K_1^F$ the symbol is $-a_i^2/2$ for $m=0$, $a_i^2/4$ for $m=\pm 1$, and $0$ for $|m|\geq 2$. For $K_2^F$ the diagonal entries are
\begin{equation}\label{eq:H2_diag}
k_2(m) = \begin{cases} -\dfrac{a_i^2}{2}\!\left(\log\dfrac{a_i}{2}+\gamma-\dfrac{i\pi}{2}\right) & m=0,\\[6pt] \dfrac{a_i^2}{4}\!\left(\log\dfrac{a_i}{2}+\gamma-\dfrac{i\pi}{2}\right)+\dfrac{a_i^2}{16} & |m|=1,\\[6pt] -\dfrac{a_i^2}{4|m|(m^2-1)} & |m|\geq 2. \end{cases}
\end{equation}

The off-diagonal block coupling resonator $j$ to resonator $i$ is computed by evaluating the relevant kernel on a $Q\times Q$ tensor-product grid of equispaced points on $\partial D_i\times\partial D_j$, with $Q\geq 4(F+4)$, and reading off all Galerkin entries simultaneously from the two-dimensional discrete Fourier transform of the resulting matrix. 

For a fixed quadrature order $Q$, the two-dimensional FFT is computed once per resonator pair and all $(2F+1)^2$ Galerkin entries are read off simultaneously, so the assembly cost $O(N^2 Q^2\log Q)$ does not grow with $F$ as long as $F \leq Q/4 - 4$. Taking $Q = O(F)$ to maintain spectral accuracy, the total cost scales as $O(N^2 F^2\log F)$.

\begin{remark}
    Since the resonators are disjoint, the kernel restricted to $\partial D_i\times\partial D_j$ for $i\neq j$ is analytic and $2\pi$-periodic in each parametric variable, so the two-dimensional trapezoidal rule converges spectrally. The diagonal blocks are evaluated exactly from the Fourier symbol, so the assembly error in all three matrices is dominated by the Fourier truncation at $|m|>F$ rather than by quadrature.
\end{remark}

The asymptotic resonances \eqref{eq:log_branch}--\eqref{eq:reg_branch} from Theorem~\ref{thm:branch_scaling} are computed by a single $N\times N$ eigenvalue problem. One assembles the dense matrices $K_1^0$ and $K_2^0$ from the inter-resonator geometry, extracts $\mu_1$, $\alpha$, and the eigenvalues $\nu_2,\dots,\nu_N$ of $B = P_{V^\perp} K_2^0 P_{V^\perp}$, and evaluates \eqref{eq:log_branch} and \eqref{eq:reg_branch} directly. This amounts to $O(N^2)$ cost for the matrix assemblies. 

\begin{figure}[H]
    \centering
    
    \begin{subfigure}[b]{0.48\textwidth}
        \centering
        \includegraphics[width=\linewidth]{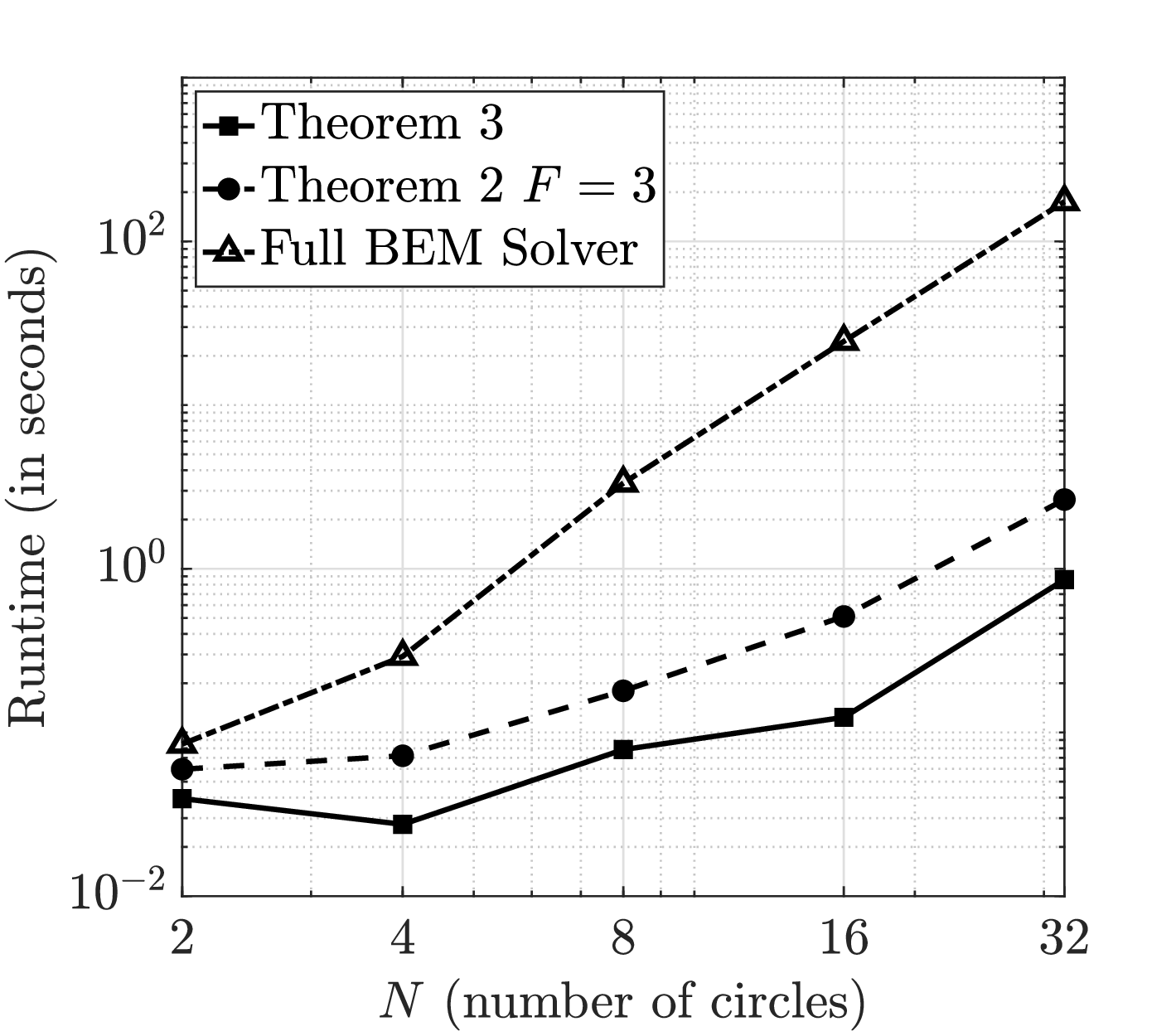}
        \caption{Runtime scaling with problem size $N$.}
        \label{fig:runtime_comparison}
    \end{subfigure}
    \hfill
    \begin{subfigure}[b]{0.48\textwidth}
        \centering
        \includegraphics[width=1.05\linewidth]{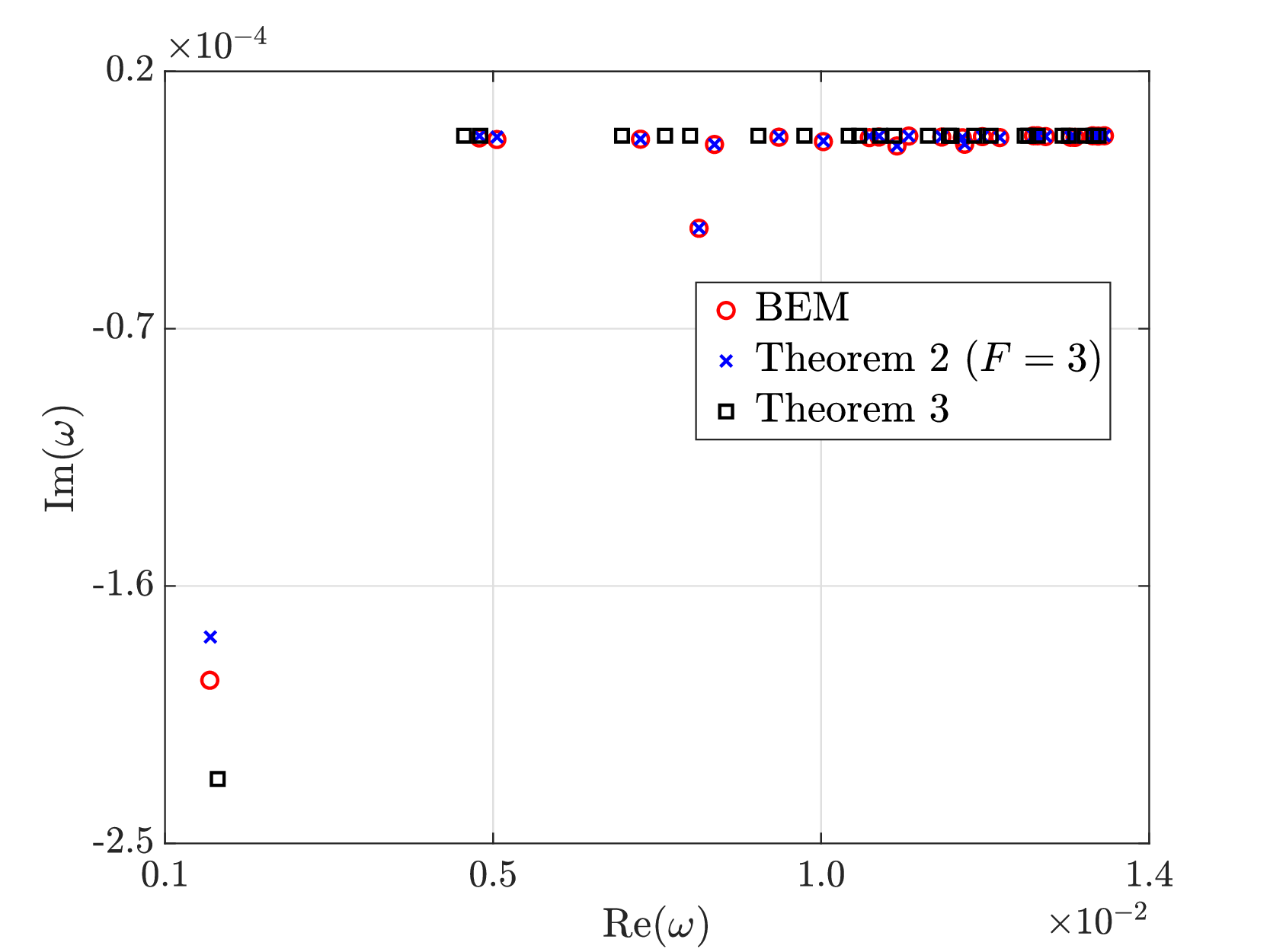}
        \caption{Validation of the computed subwavelength resonances in the complex plane for $N = 25$}
        \label{fig:resonance_comparison}
    \end{subfigure}

    \caption{Performance and accuracy of the resonance solving frameworks. (a) Computational cost scaling as the number of resonators $N$ increases. (b) Alignment of the numerical and theoretical roots in the complex plane, demonstrating the validity of the asymptotic initial seeds.}
    \label{fig:runtime_and_resonances}
\end{figure}
\begin{remark}
Figure~\ref{fig:runtime_comparison} shows that even when the asymptotic seeds from Theorem~\ref{thm:branch_scaling} are used as initial guesses for Newton refinement to full BEM accuracy (dashed line), the refinement cost grows superlinearly in $N$ due to repeated large-scale matrix assembly. Solving the effective $N\times N$ system directly (solid line) maintains near-linear scaling and achieves $\approx 100\times$ speedup at $N=32$. Computing the resonances from Theorem~\ref{thm:branch_scaling} for $1024$ resonators takes approximately three seconds..
\end{remark}
\begin{figure}[h]
    \centering
    
    \begin{subfigure}[b]{\textwidth}
        \centering
        \includegraphics[width=\textwidth]{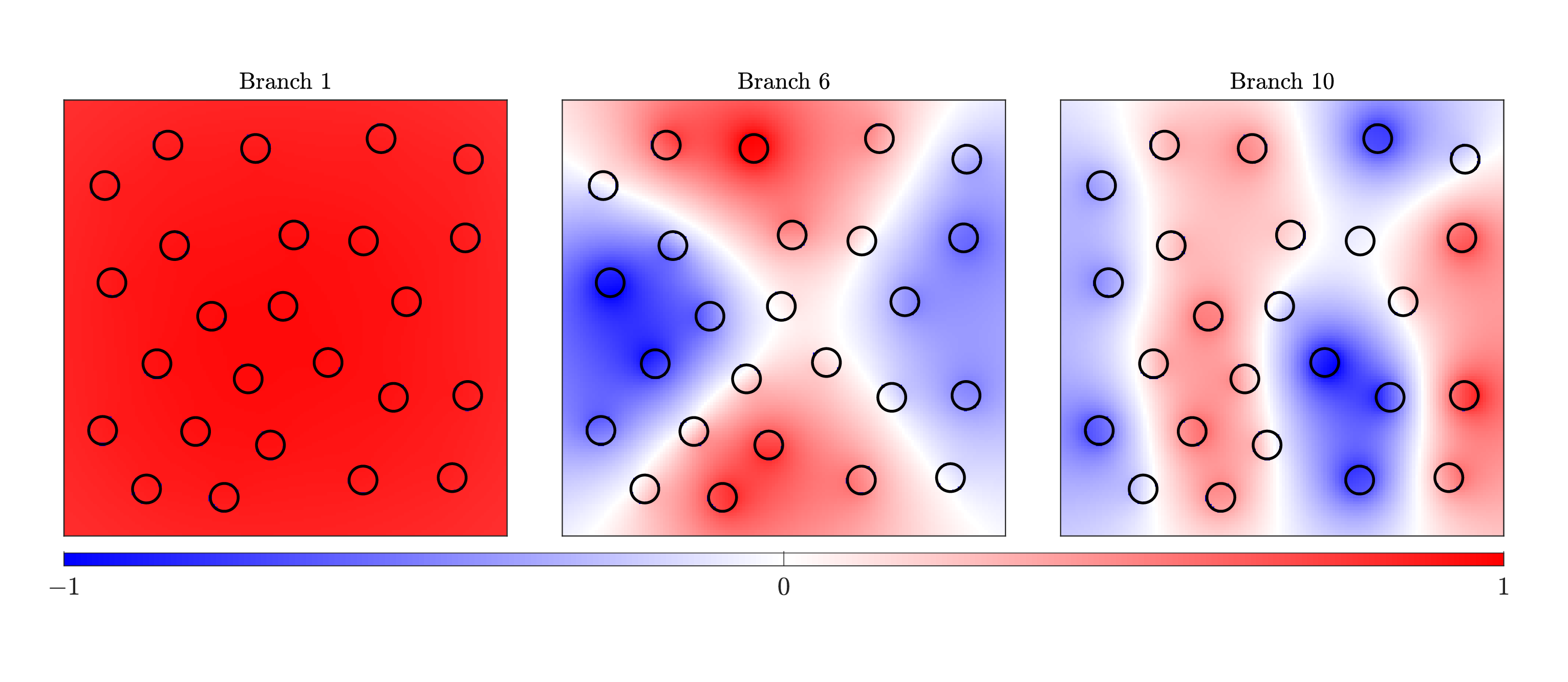}
        \vspace{-1.5cm}
        \caption{Full BEM ($F=3$).}
        \label{fig:mode_shapes_bem}
    \end{subfigure}

    \begin{subfigure}[b]{\textwidth}
        \centering
        \includegraphics[width=\textwidth]{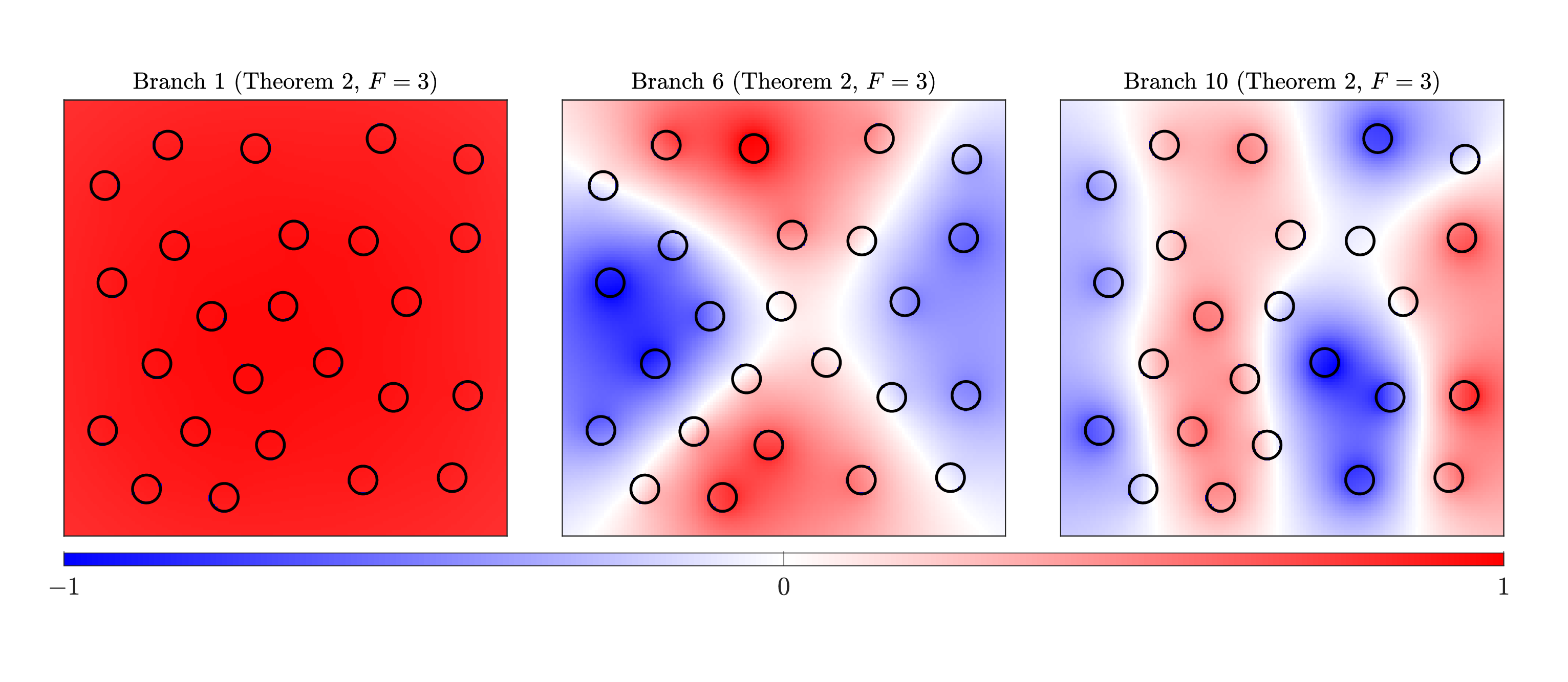}
                \vspace{-1.5cm}
        \caption{Fourier--Galerkin effective matrix (Theorem~\ref{thm:reduction}, $F=3$).}
        \label{fig:mode_shapes_eff}
    \end{subfigure}

    \begin{subfigure}[b]{\textwidth}
        \centering
        \includegraphics[width=\textwidth]{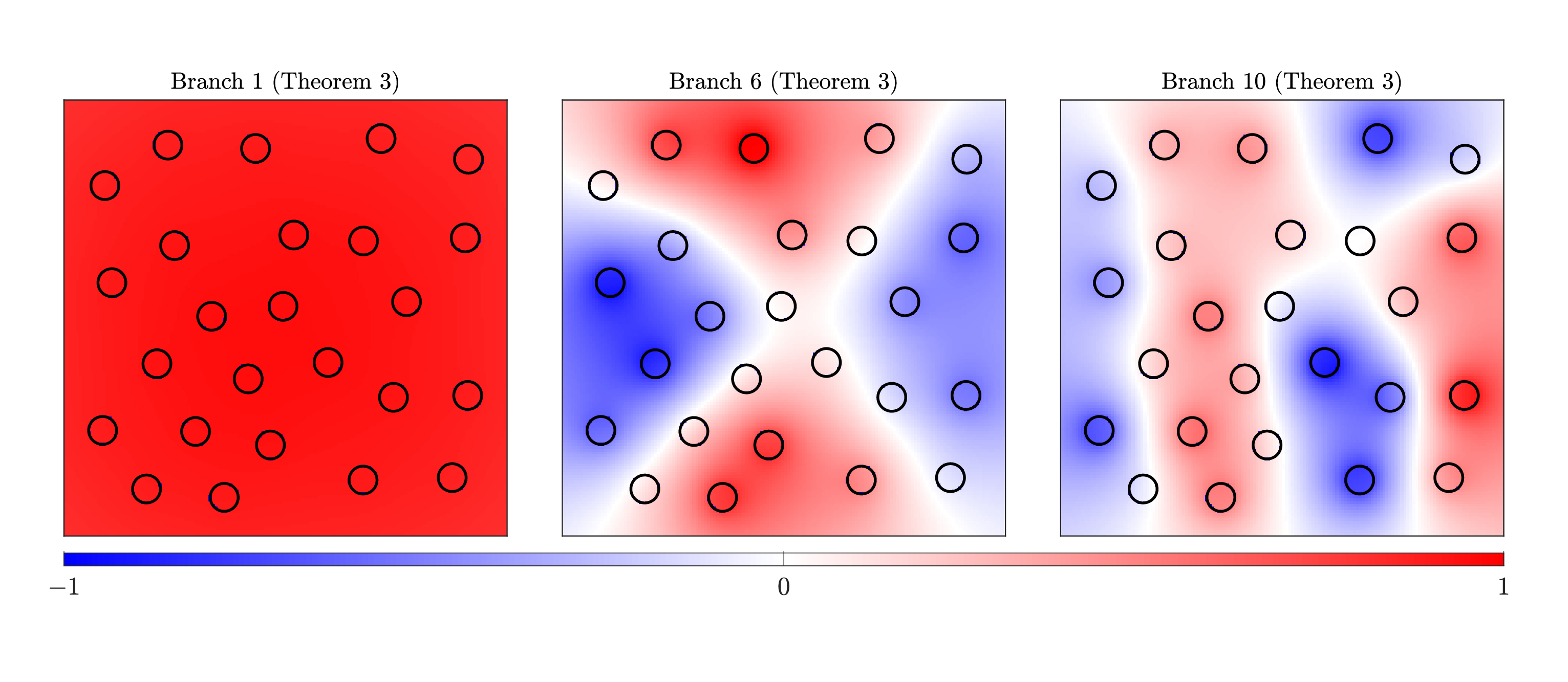}        \vspace{-1.5cm}
        \caption{Asymptotic effective matrix (Theorem~\ref{thm:branch_scaling}).}
        \label{fig:mode_shapes_asymp}
    \end{subfigure}
    
    \caption{Near-field distributions $\mathrm{Re}(u)$ for three representative subwavelength resonances of a system of $N=25$ resonators with $\delta=10^{-5}$. \emph{Left (Branch~1):} the logarithmic Minnaert resonance, acting as a macroscopic monopole. \emph{Centre (Branch~6)} and \emph{Right (Branch~10):} higher-order regular branches with strong inter-resonator interaction. Panel~(a) shows the modes resolved by the full BEM; Panel~(b) shows the Fourier--Galerkin effective matrix approximation at the same truncation order $F=3$; Panel~(c) shows the leading-order asymptotic mode shapes from Theorem~\ref{thm:branch_scaling}.}
    \label{fig:mode_shapes}
\end{figure}

\subsection{General smooth domains}
\label{sec:num2}
The assembly of the full BEM system matrix is classical and we refer to~\cite{colton_kress} for details.  We focus here on the assembly of the Galerkin matrices~\eqref{eq:galerkin_matrices}, which requires special treatment of the diagonal blocks.

For circular resonators the diagonal Galerkin blocks are circulant, and their Fourier symbols are available in closed form, so no quadrature is needed on the diagonal. For a general smooth parametrized boundary $\partial D_j$ this structure is lost: the factor $\|x'(\theta)\|$ is no longer constant, the diagonal kernel is not a convolution in $\theta-\phi$, and the Fourier symbol must be replaced by explicit numerical quadrature. The off-diagonal blocks remain smooth and are still handled by the two-dimensional trapezoidal rule at spectral cost, but each diagonal block requires a case-by-case singularity analysis.

In detail, for a smooth parametrized boundary $\partial D_j$, the Galerkin matrix entries of $C_0^F$, $K_1^F$, $K_2^F$ reduce to double integrals of
the form
\begin{equation}
\frac{1}{|\partial D_j|}
\int_0^{2\pi}\!\!\int_0^{2\pi}
\partial_{\nu_{x(\theta)}} G^{(k)}(x(\theta),x(\phi))\,
\|x'(\theta)\|\,\|x'(\phi)\|\,
e^{-im\theta}e^{in\phi}
\,d\theta\,d\phi,
\qquad |m|,|n|\leq F.
\end{equation}
Off-diagonal blocks ($i\neq j$) involve kernels evaluated on distinct, well-separated curves and are therefore smooth functions of $(\theta,\phi)$; the trapezoidal rule achieves spectral accuracy there without any special treatment. The difficulty is confined entirely to the diagonal blocks $i=j$, where $x(\theta)$ and $x(\phi)$ approach the same boundary point. We discuss each matrix in turn.

The kernel present in the matrix elements of $C_0^F$ is $\partial_{\nu_x}G^{(0)}(x,y) = (x-y)\cdot\nu_x/(2\pi|x-y|^2)$, which appears to have a Cauchy-type singularity on the diagonal. In parametric form, however, the numerator $(x(\theta)-x(\phi))\cdot\nu(\theta)$ vanishes as $O(r^2)$ as $r:=\phi-\theta\to 0$ by tangent-normal orthogonality, while the denominator $|x(\theta)-x(\phi)|^2$ vanishes as $O(r^2)$, so the two factors cancel exactly. The parametric kernel is therefore continuous on $[0,2\pi)^2$, the trapezoidal rule achieves spectral accuracy, and the diagonal entries of the Galerkin matrix are set to this limiting value explicitly.

For $K_1^F$ the kernel $\partial_{\nu_x}G^{(1)}(x,y) = -(x-y)\cdot\nu_x/(4\pi)$ is already smooth away from the diagonal. At $\phi = \theta$, tangent-normal orthogonality gives $(x(\theta)-x(\phi))\cdot\nu(\theta) = O((\phi-\theta)^2)$, so the parametric integrand vanishes at the diagonal and is continuous on $[0,2\pi)^2$. The trapezoidal rule achieves spectral accuracy with no further correction.

Finally, we tackle the assembly for the {matrix elements of $K_2^F$.} The kernel $\partial_{\nu_x}G^{(2)}$ contains the factor $(x-y)\cdot\nu_x\,\log(|x-y|/2)$. The dot-product factor vanishes as $O(r^2)$ with $r = \phi-\theta$, but this does not cancel the logarithm: the product behaves as $O(r^2\log r)$ near the diagonal, which is continuous but not smooth. Its Fourier coefficients decay only algebraically, so the trapezoidal rule loses spectral accuracy on the diagonal block. We restore it by the Kussmaul--Martensen splitting~\cite{kress2014,saranen2002periodic}. Writing $c_{\log}$ for the coefficient of the logarithmic factor in $\partial_{\nu_x}G^{(2)}$, we decompose the parametric kernel as
\begin{align}
\label{eq:splitting}
\partial_{\nu_x}G^{(2)}\,\|x'\|^2
&= \Bigl[\partial_{\nu_x}G^{(2)}\,\|x'\|^2 - c_{\log}(x(\theta){-}x(\phi))\cdot\nu(\theta)\log|\phi{-}\theta|\Bigr] \notag\\
&\quad + c_{\log}(x(\theta){-}x(\phi))\cdot\nu(\theta)\log|\phi{-}\theta|.
\end{align}
The first bracketed term in~\eqref{eq:splitting} is smooth and $2\pi$-periodic, and is integrated by the trapezoidal rule at spectral accuracy. The second line in~\eqref{eq:splitting} is further split by writing
\begin{equation}
\log|\phi-\theta| = \log\frac{|\phi-\theta|}{|2\sin\frac{\phi-\theta}{2}|} + \log\!\left|2\sin\tfrac{\phi-\theta}{2}\right|,
\end{equation}
where the first term is again smooth and handled by the trapezoidal rule, and the second has the known Fourier series~\cite{kress2014}
\begin{equation}
\log\!\left|2\sin\tfrac{\phi-\theta}{2}\right| = -\sum_{k\neq 0}\frac{1}{2|k|}e^{ik(\phi-\theta)}.
\end{equation}
Its contribution to the Galerkin matrix is computed exactly as a discrete convolution against the analytically known coefficients $\hat{g}_k = -1/(2|k|)$, with no error beyond the aliasing from truncation at $|k|\leq Q/2$. The three parts together restore spectral convergence in $Q$ for the diagonal blocks of $K_2^F$.

\begin{figure}
        \centering
        \includegraphics[width=\textwidth]{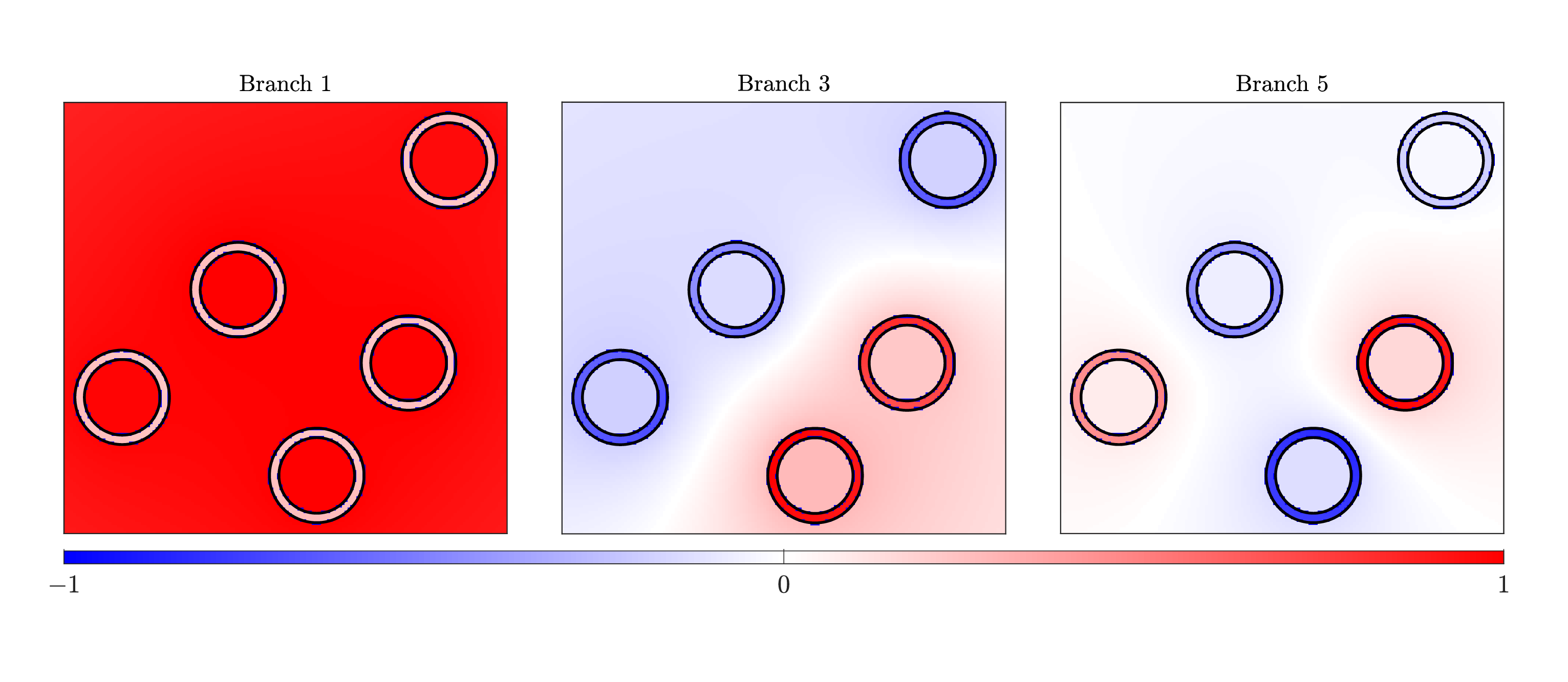}
        \vspace{-1cm}
        \caption{Illustration for selected eigenmodes for ring-shaped resonators.}
        \label{fig:modes_ring}
\end{figure}
\begin{figure}
        \centering
        \includegraphics[width=\textwidth]{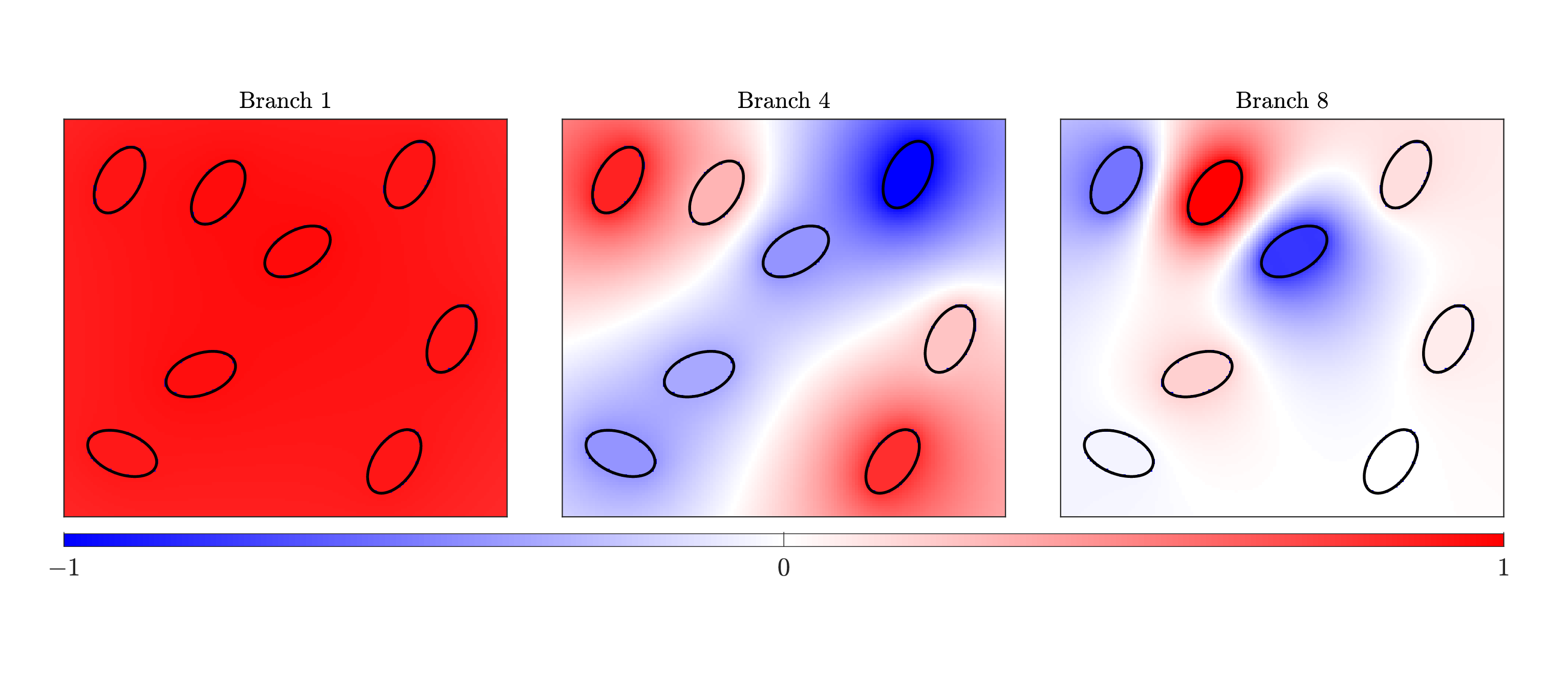}
        \vspace{-1cm}
        \caption{Illustration for selected eigenmodes for elliptical resonators.}
        \label{fig:modes_ellipse}
    \end{figure}
\begin{figure}
        \centering
        \includegraphics[width=\textwidth]{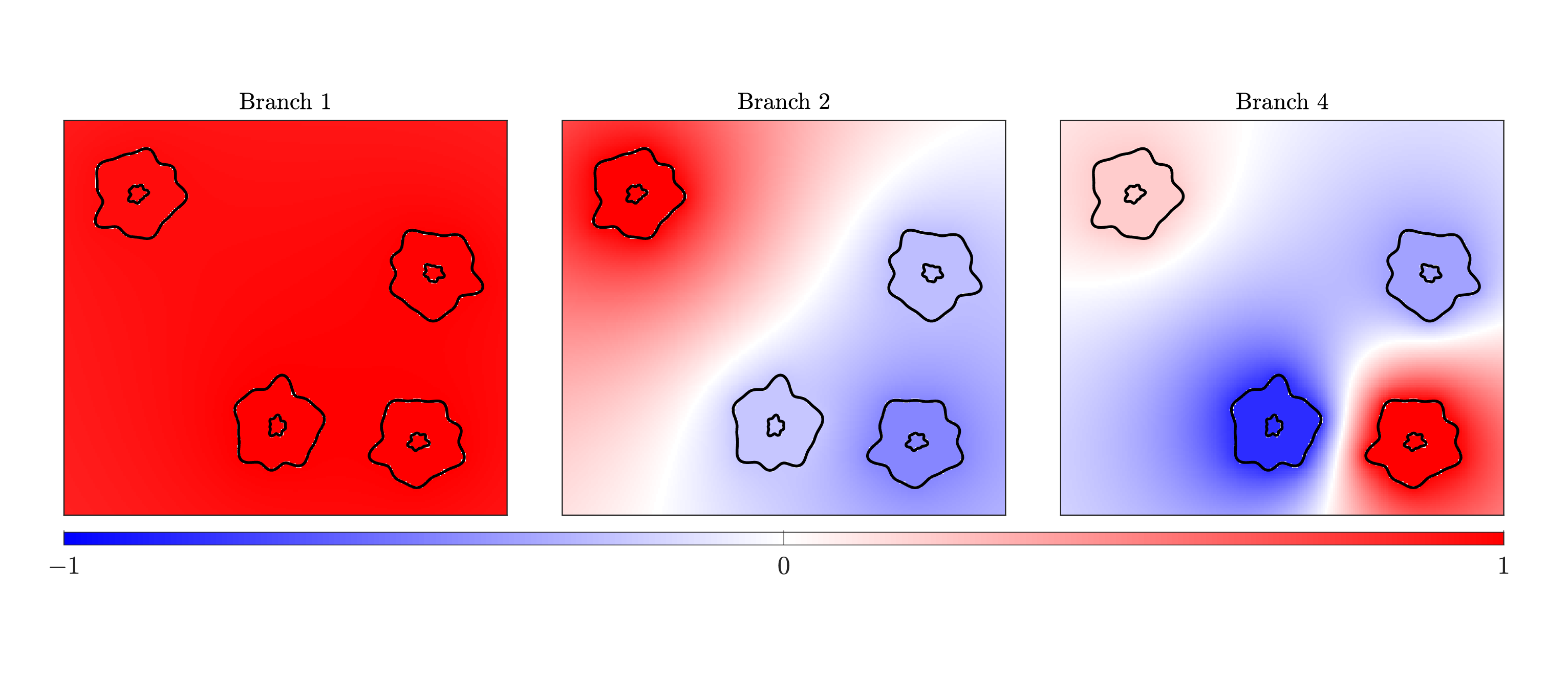}
        \vspace{-1cm}
        \caption{Illustration for selected eigenmodes for curved-ring resonators.}
        \label{fig:modes_dr}
    \end{figure}
An alternative approach that requires no kernel-specific algebra is provided directly by the \emph{Truncated Fourier Filtering} (TFF) Method introduced in~\cite{bruno2025fast}.

\begin{lemma}[\cite{bruno2025fast}]
\label{lemma:tff}
Let $g\in L^2([0,2\pi])$, $f\in C^\infty(\mathbb{R})$ be $2\pi$-periodic, $b>1$, and $q\geq bF_n$. Denote by
\begin{equation}\label{eq:trunc_fourier}
g_{F_n}(\phi) = \sum_{|k|\leq F_n} \hat{g}_k\,e^{ik\phi}, \qquad \hat{g}_k = \frac{1}{2\pi}\int_0^{2\pi} g(\phi)\,e^{-ik\phi}\,d\phi,
\end{equation}
the truncated Fourier series of $g$, and let $Q_{q,[0,2\pi]}$ denote the $q$-point trapezoidal rule on $[0,2\pi]$. Then for any $m\in\mathbb{N}$,
\begin{equation}\label{eq:tff_estimate}
\left|\int_0^{2\pi} f(\phi)\,g(\phi)\,d\phi - Q_{q,[0,2\pi]}\!\left[f(\phi)\,g_{F_n}(\phi)\right]\right| = O(F_n^{-m}).
\end{equation}
\end{lemma}
For fixed $\theta$, the inner integral in the $K_2^F$ diagonal block has the form $\int_0^{2\pi} f(\theta,\phi)\,g(\theta,\phi)\,d\phi$, where $f(\theta,\cdot) = (x(\theta)-x(\phi))\cdot\nu(\theta)\,\|x'(\theta)\|\,\|x'(\phi)\|$ is smooth and $2\pi$-periodic, and $g(\theta,\phi) = \log(|x(\theta)-x(\phi)|/2)$ is the singular factor lying in $L^2$. This is precisely the setting of Lemma~\ref{lemma:tff}, and the TFF approximation gives spectral convergence, with no analytical knowledge of the singularity structure required. 

Since the TFF approximation error $O(F_n^{-m})$ in Lemma~\ref{lemma:tff} is uniform in $\theta$ for fixed smooth $f(\theta,\cdot)$, integrating in $\theta$ against the bounded weight $\varphi_{v'}(\theta)\|x'(\theta)\|$ preserves the spectral convergence, and the full double integral inherits the same $O(F_n^{-m})$ error bound.

\begin{remark}\label{rem:resonator_shapes}
Figures~\ref{fig:modes_ring},~\ref{fig:modes_ellipse} and~\ref{fig:modes_dr} illustrate the resonance mode profiles for systems composed of ring-shaped, ellipse-shaped and curved-ring-shaped resonators, typically considered in computational work with smooth boundaries such as \cite{bruno2025fast} and \cite{fryklund2023fmm}. 
\end{remark}

\noindent \textbf{Code Availability}. MATLAB code reproducing all figures: \url{https://github.com/jinghaocao/2DFourierGalerkin}.\\
\textbf{Acknowledgment}. {The author thanks Habib Ammari for the insightful discussions. This paper was written during the author's time as a Von Karman Instructor at California Institute of Technology hosted by Oscar Bruno.}
\printbibliography
\end{document}